\documentclass[english,DIV=calc]{scrartcl}
\usepackage[utf8]{inputenc}
\usepackage[T1]{fontenc}
\usepackage{babel}
\usepackage{lmodern}
\usepackage{microtype}
\usepackage{graphicx}
\usepackage[section]{placeins}

\setkomafont{caption}{\scriptsize}

\usepackage{amsmath,amssymb,amsthm,thmtools,physics,mathtools}
\declaretheorem{theorem}
\declaretheorem[sibling=theorem]{lemma}
\declaretheorem[sibling=theorem]{proposition}
\declaretheorem[sibling=theorem]{definition}
\declaretheorem[sibling=theorem]{conjecture}

\renewcommand{\theenumi}{\alph{enumi}}
\renewcommand{\labelenumi}{(\theenumi)}

\usepackage{csquotes}
\usepackage[%
  backend=biber,
  style=numeric-comp,
  url=false,
  doi=true,
  isbn=false,
  giveninits=true,
  maxnames=6
]{biblatex}
\AtBeginBibliography{\scriptsize}

\usepackage{hyperref}

\title{Dynamical Gibbs\,--\,non-Gibbs transitions in the
  Curie--Weiss Potts model in the regime $\beta<3$ }

\author{
  Christof K\"ulske\footnote{
    Ruhr-Universit\"at Bochum, Fakult\"at f\"ur Mathematik,
    Universit\"atsstra\ss e 150, 44780 Bochum, Germany.  E-mail:
    \texttt{daniel.meissner-i4k@ruhr-uni-bochum.de},
    \texttt{christof.kuelske@ruhr-uni-bochum.de}
  }
  \and Daniel Mei\ss ner\footnotemark[\value{footnote}]
}
\date{\today}

\begin{document}
\maketitle

\begin{abstract}
  We consider the Curie--Weiss Potts model in zero external field under
  independent symmetric spin-flip dynamics.  We investigate dynamical
  Gibbs\,--\,non-Gibbs transitions for a range of initial inverse
  temperatures $\beta<3$, which covers the phase transition point
  $\beta = 4\log 2$ \parencite{ElWa90}.  We show that finitely
  many types of trajectories of bad empirical measures appear,
  depending on the parameter $\beta$, with a possibility of re-entrance
  into the Gibbsian regime, of which we provide a full description.
\end{abstract}



\section{Introduction}

\subsection{Research context}

The past years have seen progress from various directions in the
understanding of Gibbs\,--\,non-Gibbs transitions for trajectories of
measures under time-evolution, and also more general transforms of
measures.  The Gibbs property of a measure describing the state of a
large system in statistical mechanics is related to the continuity of
single-site conditional probabilities, considered as a function of the
configuration in the conditioning.  If a measure becomes non-Gibbsian,
there are internal mechanisms which are responsible for the creation
of such discontinuous dependence. This leads to the study of
\emph{hidden phase transitions}, which was started in the particular
context of renormalization group pathologies in \textcite{EnFeSo93}.

Such studies have been made for a variety of systems in different
geometries, for different types of local degrees of freedom, and under
different transformations.  Let us mention here time-evolved discrete
lattice spins \parencite{EnFeHoRe02,KiKu20}, continuous lattice spins
\parencite{KuRe06,EnKuOpRu10}, time-evolved models of point particles
in Euclidean space \parencite{JaKu17}, and models on trees
\parencite{EnErIaKu12}.  For a discussion of non-Gibbsian behavior of
time-evolved lattice measures in regard to the approach to a (possibly
non-unique) invariant state under dynamics, see \cite{JaKu19}, for
relevance of non-Gibbsianness to the infinite-volume Gibbs variational
principle (and its possible failure) see \cite{KuLeRe04,LaTa20}.
For recent developments for one-dimensional long-range systems, and
the relation between continuity of one-sided (vs. two-sided)
conditional probabilities see
\cite{EnLe17,BeFeVe19,BeStCo18,BiEnEnLe18}.

In the present paper we are aiming to contribute to the understanding
of Gibbs\,--\,non-Gibbs transformations for mean-field models, in the
sense of the sequential Gibbs property
\parencite{KuLe07,ErKu10,FeHoMa13,JaKuRuWe14,HoReZu15,KiKu19,FeHoMa14,HeKrKu19}.
Usually there is a somewhat incomplete
picture for lattice models, due to the difficulty to find sharp
critical parameters.  Mean-field models on the other hand are
often “solvable” in terms of variational principles which arise from
the large deviation formalism, while the remaining model-dependent
task to characterize the minimizers and understand the corresponding
various bifurcations can be quite substantial.  We choose to work for
our problem in the so-called two-layer approach, in which one needs to
understand the parameter dependence of the large-deviation functional
of a conditional first-layer system. In this functional the
conditioning provides an additional parameter given by an empirical
measure on the second layer. This is more direct than working in the
Lagrangian formalism on trajectory space, which would provide
additional insights on the nature of competing histories that explain
the current state of the system at a discontinuity point
\parencite{EnFeHoRe10,ErKu10,ReWa14,KrReZu17}.

Compared to the Curie--Weiss Ising model, the Fuzzy Potts model and the
Widom-Rowlinson models, we find in the present analysis of the
time-evolved Curie--Weiss Potts model significantly more complex
transition phenomena, see Theorem~\ref{thm:time_evolution} and
Figure~\ref{fig:dynamical_phase_diagram}. This has to be expected as
already the behavior of the fully non-symmetric static model is subtle
\parencite{KuMe20}.  It forces us to make use of the computer
for exact symbolic computations, in the derivation of the transition
curves (BU, ACE and TPE in Figure~\ref{fig:dynamical_phase_diagram},
discussed in Sects.~\ref{sec:bu}, \ref{sec:ace} and \ref{sec:tpe}),
along with some numerics for our bifurcation analysis.  We believe
that these tools (see page~\pageref{source_code_info}) may also be
useful elsewhere.

Now, our approach rests on singularity theory
\parencite{PoSt78,ArGuVa85,GaMaMi98,GaMaMi99,Broecker75} for
the appropriate conditional rate functional of the dynamical model.
This provides us with a four-parameter family of potentials, for a
two-dimensional state-variable taking values in a simplex.  It turns
out that the understanding of the parameter dependence of the
dynamical model is necessarily based on the good understanding of the
bifurcation geometry of the free energy landscape of the static case
for general vector-valued fields \parencite{KuMe20}.  In that
paper, which generalizes the results of \textcite{ElWa90,Wang94}, we
lay out the basic methodology. Therein we also explain the
phenomenology of transitions (umbilics, butterflies, beak-to-beak)
from which we need to build here for the dynamical problem.

As a result of the present paper we show that the unfoldings of the
static model indeed reappear in the dynamical setup, and acquire new
relevance as hidden phase transitions. It is important to note that,
in order for this to be true, we have to restrict to mid-range inverse
temperatures $\beta<3$. More work has still to be done to treat the
full range of inverse temperatures for the dynamical model, where more
general transitions seem to appear for very low temperatures.  For the
scope of the present paper, it is this close connection between the static
model \parencite{KuMe20} in fully non-symmetric external fields, and the
symmetrically time-evolved symmetric model in intermediate $\beta$
range, which is really crucial to unravel the types of trajectories of
bad empirical measures of Theorem~\ref{thm:time_evolution}.  It would
be challenging to exploit whether an analogous non-trivial connection, that
we observe for our particular model, holds for more general classes of
models.  This clearly asks for more research.

\subsection{Overview and organization of the paper}

In the present paper we study the simplest model which is, together
with its time-evolution, invariant under the permutation group with
three elements: We consider the 3-state Curie--Weiss Potts model in
zero external field, under an independent symmetric stochastic
spin-flip dynamics.  Based on previous examples \parencite{KuLe07},
one may expect loss without recovery of the Gibbs property for all
initial temperatures lower than a critical one (which then may or may
not coincide with the critical temperature of the initial model), and
Gibbsian behavior for all times above the same critical temperature.
We show that this is not the case for our model, and the behavior is
much more complicated: The trajectories of the model show a much
greater variety, depending on the initial temperature.  We find a
regime of Gibbs forever (I), a regime of loss with recovery (II) and a
regime of loss without recovery (III).
Figure~\ref{fig:intro:phase_diagram} shows the non-Gibbs region in the
two-dimensional space of initial temperature and time.  The boundary
of this non-Gibbs region consists of three different curves which
correspond to \emph{exit scenarios} of different types of \emph{bad
  empirical measures}.  Bad empirical measures are points of
discontinuity of the limiting conditional probabilities as defined in
Definition~\ref{def:model:seq_gibbs}.  Under the time evolution
$t\uparrow \infty$ (or equivalently $g_t \downarrow 0$ given by
\eqref{eq:def:g_t}) the system moves along vertical lines of fixed
$\beta$ towards the temperature axis.  Intersections with a finite
number of lines occur along this way, which are responsible for the
transitions described in our main theorem,
Theorem~\ref{thm:time_evolution}.  These additional relevant lines are
shown in Figure~\ref{fig:dynamical_phase_diagram}.
Theorem~\ref{thm:time_evolution} rests on the understanding of the
structure of stationary points of the time-dependent conditional rate
function given in Formula~\eqref{eq:hs_transform} via singularity
theory.

\begin{figure}
  \centering
  \includegraphics{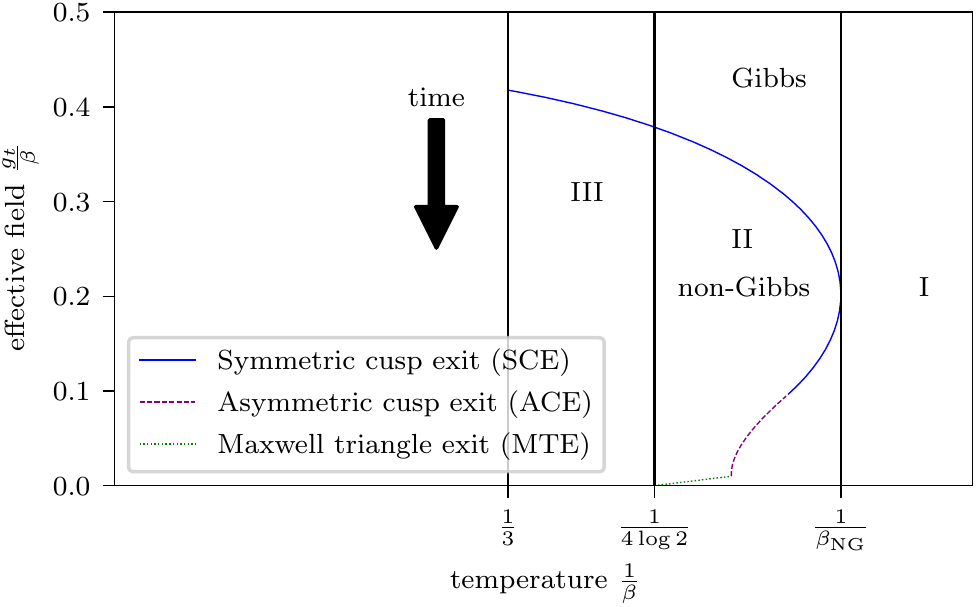}
  \caption{This figure shows the non-Gibbs region for the mid-range
    temperature regime we consider.  The boundary of this region
    consists of three different curves which correspond to exit
    scenarios of bad empirical measures.}
  \label{fig:intro:phase_diagram}
\end{figure}

It turns out that the bifurcations we encounter for general values of
the four-dimensional parameter
\((\alpha, \beta, t) \in \Delta^2 \times (0, \infty) \times (0,
\infty)\) (see \eqref{eq:unit_simplex}) are of the same types as for
the static model depending on a three-dimensional parameter.  However,
this holds only \emph{if} we restrict to mid-range inverse
temperatures \(\beta < 3\) \emph{and} to endconditionings $\alpha$
taking values in the unit simplex (and not in the full hyperplane
spanned by the simplex).  Nevertheless, in order to understand the
relevant singularities, the analysis is best done by first relaxing
the probability measure constraint on the parameter $\alpha$ and allow
it to take values in the hyperplane.  The analysis proceeds with a
description of the \emph{bifurcation set}, where the structure of
stationary points of the conditional rate function changes, and the
\emph{Maxwell set}, where multiple global minimizers appear.  To pick
from these transitions the ones which are relevant to the problem of
sequential Gibbsianness and visible on the level of bad empirical
measures, we have to take the probability measure constraint for
\(\alpha\) into account.  This step is neither necessary in the static
Potts nor in the dynamical symmetric Ising model.  The lines
\emph{Symmetric cusp exit (SCE), Asymmetric cusp exit (ACE), Triple
  point exit (TPE) and Maxwell triangle exit (MTE)} depicted in the
full phase diagram in Figure~\ref{fig:dynamical_phase_diagram} are
examples of such exit scenarios.  For those lines there is an exit of
a certain particular critical value of $\alpha$ from the unit simplex
(observation window).  The detailed dynamical phase diagram in
Figure~\ref{fig:dynamical_phase_diagram} shows more information about
the transitions during time evolution.  Preliminary investigations
show that the structural similarity with the static case may no longer
be valid in the regime $\beta>3$.  Therefore we leave the region of
very low temperatures for future research.

We describe the model we are considering together with its
time-evolution in Sect.~\ref{sec:intro:model} where we also define
what we mean by Gibbsianness (or the sequential Gibbs property).  In
Sect.~\ref{sec:main_result} we present our main theorem and describe
the transitions of the sets of bad empirical measures as a function of
the parameters \(\beta\) and \(t\).  We will establish the connection
between the analysis of the potential function
\(G_{\alpha, \beta, t}\) and the Gibbs property of the time-evolved
model in Sect.~\ref{sec:gibbs_vs_potential}.  The analysis of the
potential function using the methods of singularity theory is then carried out
in the Sects.~\ref{sec:loss_and_recovery} and \ref{sec:loss_only}.

\subsection{The model and sequential Gibbsianness}
\label{sec:intro:model}

We consider the mean-field Potts model with three states in vanishing
external field under an independent symmetric spin-flip dynamics.  The
space of configurations in finite-volume \(n \ge 2\) is defined as
\(\Omega_n = \{1, 2, 3\}^n\) and the Hamiltonian of the initial model
is
\begin{equation}
  \label{eq:intro:model:hamiltonian}
  H_n(\sigma) = -\frac{1}{2n} \sum_{i, j=1}^n \delta_{\sigma_i, \sigma_j}.
\end{equation}
So at time \(t = 0\) the distribution of the model is given by
\begin{equation}
  \label{eq:intro:model:time_zero_measure}
  \mu_{n, \beta}(\sigma) = \frac
  {e^{-\beta H_n(\sigma)}}
  {\sum_{\tilde\sigma\in\Omega_n}  e^{-\beta H_n(\tilde\sigma)}}.
\end{equation}
We consider a rate-one symmetric spin-flip time-evolution in terms of
independent Markov chains on the sites with transition probabilities
\begin{equation}
  \label{eq:intro:model:trans_prob}
  p_t(a, b) = \frac{e^{g_t 1_{b=a}}}{e^{g_t} + 2}
\end{equation}
from state \(a\) to \(b\) where
\begin{equation}
  \label{eq:def:g_t}
  g_t = \log\frac{1+2e^{-3t}}{1-e^{-3t}}.
\end{equation}
We are interested in the Gibbsian behavior of the time-evolved measure
\begin{equation}
  \label{eq:intro:model:time_evolved_measure}
  \mu_{n, \beta, t}(\eta) =
  \sum_{\sigma\in\Omega_n} \mu_{n, \beta}(\sigma) \prod_{i=1}^n p_t(\sigma_i, \eta_i).
\end{equation}
 The \emph{unit simplex}
\begin{equation}
  \label{eq:unit_simplex}
  \Delta^2 = \{\nu \in \mathbb{R}^3 \:|\: \nu_i \ge 0, \sum_{i=1}^3 \nu_i = 1\}
\end{equation}
contains the empirical distributions of spins.  By Gibbsian behavior
we mean the existence of limiting conditional probabilities in the
following sense.

\begin{definition}
  \label{def:model:seq_gibbs}
  The point \(\alpha\) in \(\Delta^2\) is called a \emph{good point}
  if and only if the limit
  \begin{equation}
    \gamma_{\beta, t}(\cdot|\alpha) :=
    \lim_{n\to\infty} \mu_{n, \beta, t}(\cdot|\eta_{n,2},\dots,\eta_{n,n})
  \end{equation}
  exists for every family \(\eta_{n, k} \in \{1, 2, 3\}\) with
  \(n\ge 2\) and \(2 \le k \le n\) such that
  \begin{equation}
    \lim_{n\to\infty} \frac{1}{n-1} \sum_{k=2}^n \eta_{n,k} = \alpha.
  \end{equation}
  We call \(\alpha\) bad, if it is not good.  The model
  \(\mu_{\beta, t}\) is called \emph{sequentially Gibbs} if all
  \(\alpha\) in the unit simplex \(\Delta^2\) are good points.
\end{definition}

\section{Dynamical Gibbs\,--\,non-Gibbs transitions: main result}
\label{sec:main_result}

Our main result on the dynamical Gibbs\,--\,non-Gibbs transitions in
the high-to-intermediate temperature regime for the initial inverse
temperature $\beta<3$ is as follows.  This temperature regime ranges
from high temperature, covering the phase transition temperature
(Ellis-Wang inverse temperature $\beta = 4\log 2$), up to the elliptic
umbilic point $\beta=3$ (where the central stationary point of the
time-zero rate function in zero external field changes from minimum to
maximum).

Essential parts of the structure of the trajectories of dynamical
transitions as a function of time $t$ in the regime $\beta<3$ remain
unchanged over the three inverse-temperature intervals I, II and III,
which were already visualized in Figure~\ref{fig:intro:phase_diagram}.
The type of transitions can be understood as deformations of the
sequences of transitions found in the static Potts model in general
vector-valued fields analyzed in \cite{KuMe20}, where in that
case only the one-dimensional parameter $\beta$ was varied.  Observe
that however, the dynamical transitions we describe here, do not
necessarily occur in a monotonic order with respect to what is seen in
the static model under temperature variation.  This is for instance
(but not only) apparent in the phenomenon of recovery of Gibbsianness.
At very low temperatures (\(\beta > 3\)) different bifurcations seem
to occur which will be left for future research.  While reading the
following theorem it is useful to have
Figure~\ref{fig:dynamical_phase_diagram} in mind as the inverse
temperatures and transition times are related to the lines depicted in
the \emph{dynamical phase diagram}.

\begin{figure}
  \centering
  \includegraphics{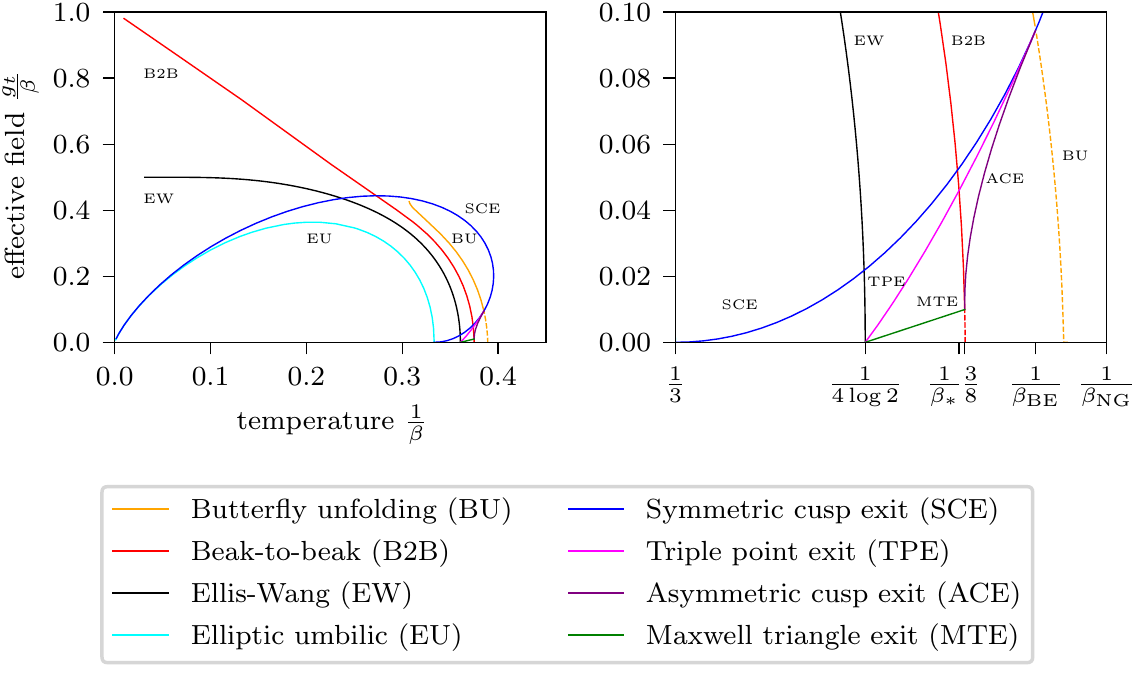}
  \caption{This figure shows the dynamical phase diagram which
    displays all lines in the two-dimensional space of
    \(\frac{1}{\beta}\) and \(\frac{g_t}{\beta}\) at which the
    structure of the bifurcation set slice or the Maxwell set slices
    changes.  We have also marked the six important temperatures in
    the magnified plot on the right.}
  \label{fig:dynamical_phase_diagram}
\end{figure}

\begin{theorem}\label{thm:time_evolution}
  Consider the time-evolved Curie--Weiss Potts model given by
  (\ref{eq:intro:model:hamiltonian}\,--\,\ref{eq:intro:model:time_zero_measure})
  in zero external field, for initial inverse temperature $\beta>0$
  and at time $t>0$ under the symmetric spin-flip dynamics
  (\ref{eq:intro:model:trans_prob}\,--\,\ref{eq:intro:model:time_evolved_measure}).
  Then the following holds.
  \begin{enumerate}
    \renewcommand{\theenumi}{\Roman{enumi}}
    \renewcommand{\labelenumi}{(\theenumi)}
  \item For $\beta< \beta_{\mathrm{NG}}\approx 2.52885$ the
    time-evolved model is sequentially Gibbs for all $t>0$.
  \item For \(\beta_{\mathrm{NG}} < \beta < 4\log 2\) the time-evolved
    model loses and then recovers the Gibbs property at sharp
    transition times.  More precisely, there exist
    \(\beta_{\mathrm{BE}} < \beta_{\ast}\) in this interval such that
    the following types of trajectories of sets of bad empirical
    measures occur:
    \begin{enumerate}
    \item For \(\beta < \beta_{\mathrm{BE}}\) the bad empirical
      measures are given by three symmetric straight lines which are
      first growing with time from the midpoints of the simplex edges
      towards the center, then shrinking with time again.
    \item For \(\beta_{\mathrm{BE}} < \beta < \frac{8}{3}\) the bad
      empirical measures are given by three symmetric straight lines
      in a first time interval
      \(t_{\mathrm{NG}}(\beta) < t < t_{\mathrm{BU}}(\beta)\).  For a
      second time interval
      $t_{\mathrm{BU}}(\beta)<t<t_{\mathrm{TPE}}(\beta)$, the set of
      bad empirical measures consists of three symmetric Y-shaped sets
      not touching.  For
      $t_{\mathrm{TPE}}(\beta)<t < t_{\mathrm{ACE}}(\beta)$ the set of
      bad empirical measures consists of six disconnected arcs.  For
      \(t > t_{\mathrm{ACE}}(\beta)\) the system is Gibbsian again.
    \item For \(\frac{8}{3} < \beta < \beta_\ast\) and
      \(t_{\mathrm{NG}}(\beta) < t < t_{\mathrm{BU}}(\beta)\) the bad
      empirical measures consist of three symmetric straight lines.
      For $t_{\mathrm{BU}}(\beta)<t<t_{\mathrm{TPE}}(\beta)$, the set
      of bad empirical measures consists of three Y-shaped sets not
      touching.  For
      $t_{\mathrm{TPE}}(\beta)<t < t_{\mathrm{B2B}}(\beta)$ the set of
      bad empirical measures consists of six disconnected arcs.  For
      \(t_\mathrm{B2B}(\beta) < t < t_{\mathrm{MTE}}(\beta)\) the set
      of bad empirical measures consists of three disconnected arcs.
      For \(t > t_{\mathrm{MTE}}(\beta)\) the system is Gibbsian
      again.  The inverse temperature \(\beta_\ast\) is given by the
      intersection point of the two lines B2B and TPE in
      Figure~\ref{fig:dynamical_phase_diagram}.
    \item For \(\beta_\ast < \beta < 4\log 2\) and
      \(t_{\mathrm{NG}}(\beta) < t < t_{\mathrm{BU}}(\beta)\) the bad
      empirical measures consist of three symmetric straight lines.
      For $t_{\mathrm{BU}}(\beta)<t<t_{\mathrm{B2B}}(\beta)$, the set
      of bad empirical measures consists of three Y-shaped sets not
      touching.  For
      $t_{\mathrm{B2B}}(\beta)<t < t_{\mathrm{TPE}}(\beta)$ the set of
      bad empirical measures consists of a triangle with curved edges
      and three symmetric straight lines attached.  For
      \(t_\mathrm{TPE}(\beta) < t < t_{\mathrm{MTE}}(\beta)\) the set
      of bad empirical measures consists of three disconnected arcs.
      For \(t > t_{\mathrm{MTE}}(\beta)\) the system is Gibbsian
      again.
    \end{enumerate}
  \item For \(4\log 2 < \beta < 3\) the time-evolved model loses the
    Gibbs property without recovery at a sharp transition time and the
    set of bad empirical measures has the following structure: For
    \(t \le t_{NG}(\beta)\) the time-evolved model is Gibbsian. For
    \(t_{NG}(\beta) < t < t_{\mathrm{BU}}(\beta)\) the bad empirical
    measures are given by three symmetric straight lines which are
    growing with time and become Y-shaped sets for
    \(t_{\mathrm{BU}}(\beta) < t < t_{\mathrm{B2B}}(\beta)\).  For
    \(t_{\mathrm{B2B}}(\beta) < t < t_{\mathrm{EW}}(\beta)\) the sets
    then touch and form one connected component consisting of a
    central triangle with three straight lines attached to the
    vertices.  The central triangle then shrinks to a point at
    \(t = t_{\mathrm{EW}}(\beta)\) and the bad empirical measures are
    given by three symmetric straight lines which meet in the simplex
    center for all \(t > t_{\mathrm{EW}}(\beta)\).
  \end{enumerate}
\end{theorem}

The meaning and computation of these lines are discussed in
Sects.~\ref{sec:loss_and_recovery} and \ref{sec:loss_only}.  While
only the three lines SCE, ACE and MTE appear as part of the boundary
line of the non-Gibbs region, the other lines are relevant for
structural changes of the set of bad empirical measures.  There are
lines which are explicit in the sense that they are given in terms of
zeros of one-dimensional non-linear functions, for example, the entry
time $t_{\mathrm{NG}}(\beta)$ (formula \eqref{eq:t_NG}) or the
butterfly unfolding time \(t_{\mathrm{BU}}(\beta)\)
(Formula~\eqref{eq:bu:t}).  The least explicit lines are the MTE and
TPE lines which involve a Maxwell set computation, the most explicit
line is SCE which is given in parametric form
\(s\mapsto (\beta(s), g_t(s))\) as described in
Proposition~\ref{prop:sce}.  Figure~\ref{fig:bad_measures} gives a
graphical overview of the possible types of sequences of bad empirical
measures with increasing time for the different temperature regimes.
There is an even more detailed graphic that illustrates all the
transitions involved in the bifurcation set as well as in the Maxwell
set.  You can find this graphic in the electronic supplemental
material (ESM) under the filename \texttt{detailed\_overview.pdf}.

\begin{figure}
  \centering
  \includegraphics{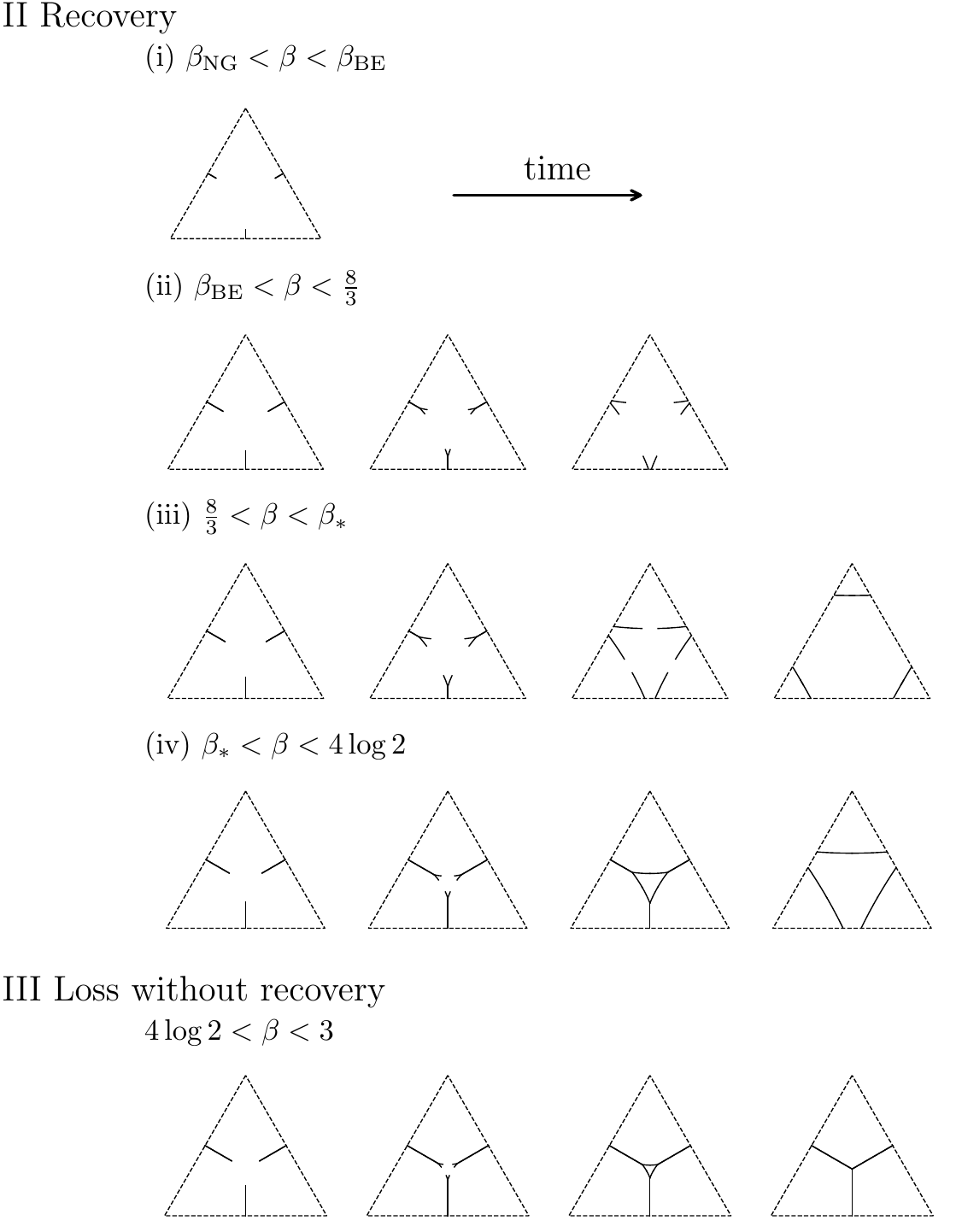}
  \caption{These are the typical sequences of bad empirical measures
    \(\alpha\) for the inverse temperature regimes described in
    Theorem~\ref{thm:time_evolution}.  With increasing time, you can
    observe the structural change of the set of bad empirical measures
    as it passes the various transition times.  For example in (II.ii)
    straight lines enter the simplex, become non-touching Y-shaped
    sets at the butterfly transition time \(t_{\mathrm{BU}}(\beta)\)
    and move out of the simplex.  The midpoints of the Y-shaped sets
    exit at \(t_{\mathrm{TPE}}(\beta)\) and the set leaves the simplex
    completely at \(t_{\mathrm{ACE}}(\beta)\).  In (II.iii) the
    midpoints of the Y-shaped sets leave the unit simplex at
    \(t_{\mathrm{TPE}}(\beta)\) and the two respective arcs connect at
    the beak-to-beak transition time \(t_{\mathrm{B2B}}(\beta)\).  The
    remaining three arcs move towards the corners and leave the unit
    simplex at \(t_{\mathrm{MTE}}(\beta)\).  The exit of the midpoints
    of the Y-shaped sets and the connection of the six arcs occurs in
    reversed order in the next row (II.iv).  In (III) the central
    triangle shrinks to a point and forms the star-like set that
    remains in the simplex forever.}
  \label{fig:bad_measures}
\end{figure}

\section{Infinite-volume limit of conditional probabilities}
\label{sec:gibbs_vs_potential}

The existence of the infinite-volume limit of the conditional
probabilities, that is, the question of sequential Gibbsianness, can be
transformed into an optimization problem of a certain potential
function.  As the parameters \((\beta, t)\) are fixed throughout this
section let us write \(\mu_n\) for the measure \(\mu_{n, \beta, t}\).

\begin{theorem}
  \label{thm:inf_vol_cond_prob}
  Suppose the Hubbard-Stratonovi\v{c} (HS) transform
  \(G_{\alpha, \beta, t}\colon \mathbb{R}^3 \to \mathbb{R}\) given by
  \begin{equation}
    \label{eq:hs_transform}
    G_{\alpha, \beta, t}(m) =
    \frac{1}{2} \beta \langle m, m \rangle - \sum_{b=1}^3 \alpha_b \log \sum_{a=1}^3 e^{\beta m_a + g_t 1_{a=b}}
  \end{equation}
  has a unique global minimizer, then $\alpha$ is a good point, that
  is, the infinite-volume limit of the conditional probabilities
  $\mu_n(\cdot|\alpha_n)$ with $\alpha_n\to\alpha$ exists
  independently of the choice of \((\alpha_n)\).
\end{theorem}

  The idea of the proof goes as follows:
  We can rewrite the conditional probabilities
  \(\mu_n(\cdot|\alpha_n)\) in terms of an expected value with respect
  to a disordered mean-field Potts model \(\bar\mu_n\) (see
  Lemma~\ref{lem:cond_first_layer}).  Thus, we have to study the weak
  convergence of \(L_n\), where \(L_n\) is the empirical distribution
  of the spins \(\sigma_2, \dots, \sigma_n\). Note that this is
  equivalent to the weak convergence of
  \(\frac{W}{\sqrt{\beta(n-1)}} + L_n\) with some independent standard
  normal variable \(W\).  Because of the representation of the
  distribution of \(\frac{W}{\sqrt{\beta(n-1)}} + L_n\) in terms of
  the function \(G_{\alpha_n, \beta, t}\)
  (Lemma~\ref{lem:gauss-noise-repr}), we can prove the theorem by an
  asymptotic analysis of integrals of the form
  \begin{equation}
    \int_{\mathbb{R}^3} f(m) e^{-(n-1) G_{\alpha_n,\beta, t}(m)} \:\mathrm{d}m
  \end{equation}
  as was done by \textcite{ElWa90}.
So it suffices to prove the Lemmata~\ref{lem:cond_first_layer}
and \ref{lem:gauss-noise-repr}.
A point is good if the respective random field model shows no phase
transition, that is, the law of large numbers holds. To be precise, we
have the following representation:

\begin{lemma}
  \label{lem:cond_first_layer}
  The finite-volume conditional probabilities are given by
  \begin{equation}
    \mu_{n}(\eta_1|\eta_2,\dots,\eta_n) = \bar\mu_{n}[\eta_2, \dots, \eta_n](f_n^{\eta_1})
  \end{equation}
  where
  \begin{equation}
    f_n^{\eta_1}(\sigma_2, \dots, \sigma_n) = \frac{\sum_a \exp\left(
        \frac{\beta}{n} \sum_{i=2}^n 1_{\sigma_i = a}
      \right) p_t(a, \eta_1)}{
      \sum_a \exp\left(
        \frac{\beta}{n} \sum_{i=2}^n 1_{\sigma_i = a}
      \right)}
  \end{equation}
  and \(\bar\mu_n\) is a quenched random field Potts model
  \begin{equation}
    \bar\mu_n[\eta_2, \dots, \eta_n](\sigma_2, \dots, \sigma_n) =
    \frac{
      \exp(\frac{\beta}{2n} \sum_{i,j=2}^n 1_{\sigma_i = \sigma_j}) \prod_{i=2}^np_t(\sigma_i, \eta_i)
    }{
     \sum_{\tilde\sigma_2, \dots, \tilde\sigma_n} \exp(\frac{\beta}{2n} \sum_{i,j=2}^n 1_{\tilde\sigma_i = \tilde\sigma_j}) \prod_{i=2}^np_t(\tilde\sigma_i, \eta_i)
    }
  \end{equation}
\end{lemma}

\begin{proof}
  The proof follows from explicit computations with conditional
  probabilities.
\end{proof}

This representation of the conditional probabilities transforms the
problem of understanding bad points to the analysis of disordered mean-field
models and their phase transitions. This analysis is done using
the Hubbard-Stratonovi\v{c} transformation which is successfully used
for many models \parencite{ElNe78,ElWa90,KuLe07}.

\begin{lemma}
  \label{lem:gauss-noise-repr}
  Write
  \begin{equation}
    L_n = \frac{1}{n-1} \sum_{i=2}^n \delta_{\sigma_i}
  \end{equation}
  for the empirical measure of \(n-1\) spins with law
  \(\bar\mu_n[\eta_2, \dots, \eta_n] \circ L_n^{-1}\).  Furthermore,
  let \(W\) be a standard normal random vector independent of \(L_n\).
  The distribution of \(W/\sqrt{\beta(n-1)} + L_n\) has a density
  proportional to \(e^{-(n-1) G_{\alpha_n, \beta, t}}\) with respect
  to Lebesgue measure.
\end{lemma}

\begin{proof}
  Denote by \(\sigma_2, \dots, \sigma_n\) independent
  \(\{1, 2, 3\}\)-valued random variables each distributed according
  to \(p_t(\dd{\sigma_i},\eta_i)\) with a fixed boundary configuration
  \(\eta_2, \dots, \eta_n\) with empirical measure \(\alpha_n\).  We
  denote the expectation with respect to this distribution by
  \(\mathbb{E}\).  Then in order to calculate the distribution of
  \begin{equation}
    \frac{W}{\sqrt{\beta(n-1)}} + L_n = \frac{W}{\sqrt{\beta(n-1)}} +
    \frac{1}{n-1}\sum_{i=2}^n \delta_{\sigma_i}
  \end{equation}
  we calculate for every bounded continuous function \(f\) the
  expectation
  \begin{equation}
    \label{eq:expectation}
    \frac{(2\pi)^{-\frac{3}{2}}}{Z_n}
    \mathbb{E}\left[ \int f\left( w/\sqrt{\beta(n-1)} + L_n \right) e^{-\frac{\Vert w\Vert^2}{2} +
      \frac{\beta}{2}(n-1) \Vert L_n\Vert^2} \dd{w} \right]
  \end{equation}
  Now we apply the transformation \(m = w/\sqrt{\beta(n-1)} + L_n\)
  and obtain
  \begin{equation}
    \frac{(2\pi)^{-\frac{3}{2}}}{Z_n}
    \mathbb{E} \left[ \int f(m) \exp\left(
      -(n-1) \frac{\beta}{2} \Vert m\Vert^2 +
      (n-1) \beta \langle m, L_n\rangle
    \right) \dd{m} \right]
  \end{equation}
  In order to complete the proof, we have to calculate the expectation
  \begin{equation}
    \begin{split}
      \mathbb{E}[\exp((n-1)\beta \langle m, L_n \rangle)] &= \prod_{i=2}^n \mathbb{E}[
      \exp(\beta m_{\sigma_i})]
      \\
      &= \prod_{i=2}^n \sum_{a=1}^3 \frac{e^{\beta m_a + g_t 1_{\eta_i=a}}}{e^{g_t} + 2}
      \\
      &= \frac{1}{(e^{g_t} + 2)^{n-1}} \prod_{i=2}^n \sum_{a=1}^3 e^{\beta
        m_a + g_t 1_{\eta_i=a}}
    \end{split}
  \end{equation}
  Now we take the logarithm to raise the expression back into the
  exponent again. So the expected value~\eqref{eq:expectation} of the
  bounded continuous function~\(f\) is equal to the following up to a
  normalizing constant:
  \begin{equation}
  \int f(m) \exp\left(
    -(n-1) \frac{\beta}{2} \Vert m\Vert^2 +
    \sum_{i=2}^n \log \sum_a e^{\beta m_a + g1_{\eta_i=a}}
  \right) \dd{m}
\end{equation}
We can now identify \(G_{\alpha_n,\beta,t}\) in the exponent using that
\begin{equation}
  \begin{split}
  \sum_{i=2}^n \log \sum_{a=1}^3 e^{\beta m_a + g_t 1_{\eta_i=a}}
  &= (n-1) \sum_{b=1}^3 \frac{1}{n-1} \sum_{i=2}^n 1_{\eta_i=b}
  \log \sum_{a=1}^3 e^{\beta m_a + g_t 1_{b=a}}
  \\
  &= (n-1) \sum_{b=1}^3 \alpha_n(b) \log \sum_{a=1}^3 e^{\beta m_a + g_t 1_{b=a}}.
  \end{split}
\end{equation}
\end{proof}

\section{Recovery of the Gibbs property}
\label{sec:loss_and_recovery}

The regime \(\beta < \frac{8}{3}\) is split into three parts given by
the intervals \((0, \beta_{\text{NG}}]\),
\((\beta_{\text{NG}}, \beta_{\text{BE}}]\) and
\((\beta_{\text{BE}}, \frac{8}{3})\).  In the first part we find that
the model is sequentially Gibbs for all times \(t > 0\) whereas in the
other two parts the system recovers from a state of non-Gibbsian
behavior.  The driving mechanism in this “recovery regime” is due to
the butterfly singularity which is already found in the static model
\parencite[see][Sect.~2.4.1]{KuMe20}.  However, in contrast to the static
model the bifurcation set might leave the unit simplex so that in
order to answer the Gibbs\,--\,non-Gibbs question the location of this
set (and the contained Maxwell set) with respect to the unit simplex
is also important.

\subsection{Elements from singularity theory}

In order to investigate the Gibbs\,--\,non-Gibbs transitions we have
to study the global minimizers of the potential
\(G_{\alpha,\beta, t}\) (Theorem~\ref{thm:inf_vol_cond_prob}).  We
will use concepts from singularity theory to derive and explain our
results.

Singularity theory allows us to understand how the stationary points
of the potential change with varying parameters.  This can be achieved
by looking at the geometry of the so-called \emph{catastrophe
  manifold}, which contains the information about the stationary points
of the potential for every possible choice of parameter values.  More
precisely, it consists of the tuples \((m, \alpha, \beta, t)\) in
\(\mathbb{R}^3 \times \Delta^2 \times (0, \infty) \times (0, \infty)\)
such that \(m\) is a stationary point of \(G_{\alpha, \beta, t}\)
given by \eqref{eq:hs_transform}.  The \emph{bifurcation set} consists
of those parameter values \((\alpha, \beta, t)\) in
\(\Delta^2 \times (0, \infty) \times (0, \infty)\) such that there
exists a degenerate stationary point \(m\) in \(\mathbb{R}^3\), that
is, a point at which the Hessian has a zero eigenvalue.  The parameter
values of the bifurcation set give rise to a partition of the
parameter space whose cells contain parameters at which the number and
nature of stationary points do not change.  Although we are only
interested in \(\alpha\) that are bad \emph{empirical} measures, hence
\emph{probability} measures, it is convenient to loosen this
constraint and consider \(\alpha\) in the hyperplane
\(H = \{m \in \mathbb{R}^3 | m_1 + m_2 + m_3 = 1\}\) into which the
unit simplex is embedded.  The following proposition is the basis for
the analysis of the bifurcation set.

\begin{proposition}
  \label{prop:sing_summary}
  Let \(\Gamma\) denote the map from
  \(\mathbb{R}^3 \times (0, \infty)\) to the space of \(3\times 3\)
  matrices with real entries \(\mathrm{Mat}(3, \mathbb{R})\) given
  by its components
  \begin{equation}
    \Gamma_{b, a}(M, t) = \frac{e^{M_a + g_t 1_{b=a}}}{\sum\limits_{c=1}^3 e^{M_c + g_t 1_{b=c}}}.
  \end{equation}
  Then we have the following:
  \begin{enumerate}
  \item Let \(\rho\) be any permutation of \(\{1, 2, 3\}\). Then
    \begin{equation}
    \label{eq:elem_sing:gamma_symmetry}
    \rho^{-1}\Gamma(M, t)\rho = \Gamma(\rho M, t)
    \end{equation}
    where we interpret the permutation \(\rho\) as a
    \(3\times 3\)-matrix and \(M\) as a column vector.  For example,
    if \(M_2 = M_3\), we find
    \(\Gamma_{3,3}(M, t) = \Gamma_{2,2}(M, t)\) and also
    \(\Gamma_{1, 2}(M, t) = \Gamma_{1, 3}(M, t)\).
  \item \(\Gamma\) maps \(\mathbb{R}^3 \times (0, \infty)\) into the
    general linear group \(\mathrm{GL}(3, \mathbb{R})\) and the
    inverse matrix of \(\Gamma(M, t)\) is given by the formulas
    \begin{align}
      \label{eq:gamma_inverse_diag}
      \Gamma^{-1}_{a,a}(M, t) &= \frac{(e^{g_t} + 1) e^{-M_a}}{ e^{2g_t} +
                                e^{g_t} - 2} \sum_{c=1}^3 e^{M_c + g_t 1_{c=a}}
      \\
      \label{eq:gamma_inverse_offdiag}
      \Gamma^{-1}_{b,a}(M, t) &= -\frac{e^{-M_b}}{e^{2g_t} + e^{g_t} - 2}
                                \sum_{c=1}^3 e^{M_c + g_t 1_{c=a}}
    \end{align}
    for two distinct elements \(a, b\) of \(\{1, 2, 3\}\).
  \item The catastrophe manifold of the HS-transform
    \(G_{\alpha, \beta, t}\) is the graph of the map
    \((m, \beta, t) \mapsto \alpha = \chi(m, \beta, t)\) given by
    \begin{equation}
      \chi(m, \beta, t) = \left(
        \sum_a m_a \Gamma^{-1}_{a, b}(\beta m, t)
      \right)_{b=1}^3
    \end{equation}
    from \(H \times (0, \infty) \times (0, \infty)\) to \(H\).  For
    \(\chi(m, \beta, t)\) to lie in the unit simplex~\(\Delta^2\) it
    is necessary (but generally not sufficient) that \(m\) lies in
    \(\Delta^2\).
  \item Consider the coordinates \((x, y, z) = \varphi_\beta(m)\) where
    \begin{equation}
      \label{eq:coordinates}
      \varphi_\beta(m) = \frac{\beta}{6} \begin{pmatrix}
        \sqrt{3} (m_3 - m_2)\\
        2m_1 - m_2 - m_3\\
        2m_1 + 2m_2 + 2m_3 - 2
      \end{pmatrix}
    \end{equation}
    for \(m \in \mathbb{R}^3\).  In these coordinates, the
    \(\beta\)-scaled simplex~\(\beta\Delta^2\) is an equilateral
    triangle in the \((x, y)\)-plane centered at the origin.  The
    Hessian matrix \(G_{\alpha, \beta, t}''(m)\) in these coordinates
    is in block diagonal form:
    \begin{equation}
      \label{eq:hessian_matrix}
      \begin{pmatrix}
        \pdv[2]{G_{\alpha, \beta, t}}{x}    &
        \pdv[2]{G_{\alpha, \beta, t}}{x}{y} &
        0
        \\
        \pdv[2]{G_{\alpha, \beta, t}}{x}{y} &
        \pdv[2]{G_{\alpha, \beta, t}}{y}    &
        0
        \\
        0 & 0 & \frac{3}{\beta}  \\
      \end{pmatrix}
    \end{equation}
    The set of degenerate stationary points is given by the solutions
    \((m, \beta, t)\) of the following equation:
    \begin{equation}
      \label{eq:elem_sing:deg_cond}
      \pdv[2]{G_{\chi(m, \beta, t), \beta, t}}{x} \pdv[2]{G_{\chi(m, \beta, t), \beta, t}}{y} - \left( \pdv[2]{G_{\chi(m, \beta, t), \beta, t}}{x}{y} \right)^2 = 0
    \end{equation}
  \end{enumerate}
\end{proposition}

Before we present the proof, let us stress the importance of this
proposition.  The matrix \(\Gamma\) naturally appears in the
derivatives of \(G_{\alpha, \beta, t}\) and has the two important
properties: Firstly, the rows of \(\Gamma\) are probability vectors
and secondly the map \(M \mapsto \Gamma(M, t)\) is compatible with the
symmetry of the model.  The fact that the catastrophe manifold is
given as a graph allows us to write the bifurcation set as the set of
\((\chi(m, \beta, t), \beta, t)\) such that
\begin{equation}
  \det G_{\chi(m, \beta, t), \beta, t}''(m) = 0
\end{equation}
with \((m, \beta, t) \in H \times (0, \infty) \times (0, \infty)\).
We can therefore take the same point of view as in the static case
\parencite[cf.][Lemma~3]{KuMe20}: We study the zeros of the
Hessian determinant as a function of \(m\) with \(\beta\) and \(t\)
fixed.  This is a two-dimensional problem since we only have to
consider points in the unit simplex~\(\Delta^2\).  Additionally,
\(\Delta^2\) is bounded so that we can simply compute the zeros of the
Hessian determinant numerically on a discretization of \(\Delta^2\) as
accurately as we want to.  In this way we can get insight into the
global shape of the bifurcation set.  It is convenient to look at this
set as composed of the \emph{bifurcation set slices} \(B(\beta, t)\),
that is, the subsets for which the parameter \((\beta, t)\) is fixed.
Figure~\ref{fig:crit_point_crit_values} shows an example of the zeros
of the Hessian determinant together with the respective image under
the map \(\chi(\cdot, \beta, t)\) for a fixed pair \((\beta, t)\).  We
now continue with the proof of the above proposition.

\begin{figure}
  \centering
  \includegraphics{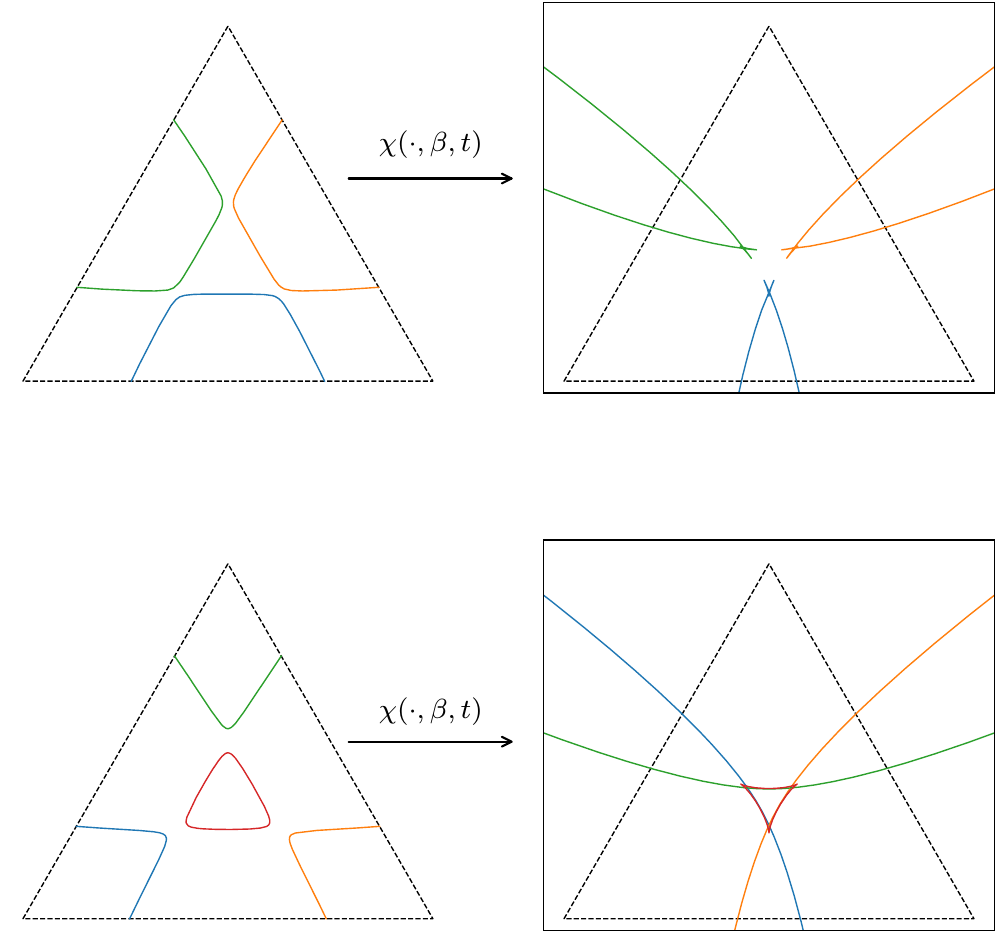}
  \caption{The left column shows the solutions to the degeneracy
    condition~\eqref{eq:elem_sing:deg_cond} for \(\beta = 2.755\),
    \(g_t = 0.5\) (above) and \(g_t = 0.45\) (below) computed using a
    uniform triangular grid.  The right column shows the image of the
    solutions under the catastrophe map \(\chi(\cdot, \beta, t)\)
    restricted to a square.  The branches of the degenerate points on
    the left and their corresponding images under
    \(\chi(\cdot, \beta, t)\) on the right are marked with the same
    color.  Note that despite the fact that the degenerate stationary
    points in the left plot lie inside of \(\Delta^2\) in the right
    plot we see that parts of the bifurcation set slice lie outside of
    the simplex.  This is a major difference to the static case.}
  \label{fig:crit_point_crit_values}
\end{figure}

\begin{proof}[Proof of Proposition~\ref{prop:sing_summary}]
  Let us prove the claims in increasing order.  Fix arbitrary
  \(M \in \mathbb{R}^3\) and positive \(t\).  The following equation
  proves \eqref{eq:elem_sing:gamma_symmetry}.
  \begin{equation}
    \Gamma_{b, a}(\rho M, t) = \frac{e^{M_{\rho(a)} + g_t 1_{b=a}}}{\sum\limits_{c=1}^3 e^{M_c + g_t 1_{b=c}}} = \frac{e^{M_{\rho(a)} + g_t 1_{\rho(b)=\rho(a)}}}{\sum\limits_{c=1}^3 e^{M_c + g_t 1_{\rho(b)=c}}} = \Gamma_{\rho(b), \rho(a)}(M, t)
  \end{equation}

  We proceed with the second point.  Note that the matrix
  \(\Gamma(M, t)\) can be written as the product \(D E\) of the
  diagonal matrix \(D = (D_{a,b})\) with entries
  \begin{equation}
    \frac{1_{a=b}}{\sum_{c=1}^3 e^{M_c + g_t 1_{c=b}}}
  \end{equation}
  for \(a, b \in \{1, 2, 3\}\) and the matrix
  \begin{equation}
    E = \begin{pmatrix}
      e^{M_1 + g_t} & e^{M_2} & e^{M_3} \\
      e^{M_1} & e^{M_2 + g_t} & e^{M_3} \\
      e^{M_1} & e^{M_2} & e^{M_3 + g_t}
    \end{pmatrix}.
  \end{equation}
  Since \(\det\Gamma(M, t) = \det(D) \cdot \det(E)\) and the
  determinant of \(D\) is clearly positive, we have to check that
  \(\det(E)\) is positive to see that \(\Gamma(M, t)\) is in the
  general linear group.  We find that the determinant of \(E\) is
  given by
  \begin{equation}
    \det(E) = e^{M_1 + M_2 + M_3}(e^{3g_t} - 3e^{g_t} + 2)
  \end{equation}
  which is clearly positive for all positive \(g_t\).

  To prove the formula for the inverse, let \(a,b\text{ and }d\) be
  pairwise different elements of \(\{1,2,3\}\).  Substituting the
  right-hand sides of
  (\ref{eq:gamma_inverse_diag}--\ref{eq:gamma_inverse_offdiag}), we
  have the following
  \begin{align*}
    \Gamma_{b,a}\Gamma^{-1}_{a,a} &= \frac{e^{g_t}+1}{ e^{2g_t} + e^{g_t} -2 } \frac{ \sum_c e^{M_c + g_t1_{c = a}}}{\sum_c e^{M_c + g_t1_{c = b}}}
    \\
    \Gamma_{b,b}\Gamma^{-1}_{b,a} &= \frac{-e^{g_t}}{ e^{2g_t} + e^{g_t} -2 } \frac{ \sum_c e^{M_c + g_t1_{c = a}}}{\sum_c e^{M_c + g_t1_{c = b}}}
    \\
    \Gamma_{b,d}\Gamma^{-1}_{d,a} &= \frac{-1}{ e^{2g_t} + e^{g_t} -2 } \frac{ \sum_c e^{M_c + g_t1_{c = a}}}{\sum_c e^{M_c + g_t1_{c = b}}}
    \\
    \Gamma_{a,a}\Gamma^{-1}_{a,a} &= \frac{(e^{g_t}+1)e^{g_t}}{ e^{2g_t} + e^{g_t} -2 }
    \\
    2\Gamma_{a,d}\Gamma^{-1}_{d,a} &= \frac{-2}{ e^{2g_t} + e^{g_t} -2 }
  \end{align*}
  Adding the right-hand sides of the first three equations yields zero
  and adding those of the last two gives one. This proves the formula
  for the inverse.

  We now prove that the catastrophe manifold is the graph of \(\chi\).
  First, let us check that the range of \(\chi\) is indeed the
  hyperplane \(H\).  Take an arbitrary point \((m, \beta, t)\) in
  \(H \times (0, \infty) \times (0, \infty)\) and let
  \(\alpha = \chi(m, \beta, t)\).
  \begin{equation}
    \sum_{b=1}^3 \alpha_b = \sum_{b=1}^3 \sum_{a=1}^3 m_a \Gamma^{-1}_{a, b}(\beta m, t)
  \end{equation}
  Since \((1, 1, 1)^{\mathrm{T}}\) is an eigenvector of
  \(\Gamma(\beta m, t)\) for the eigenvalue 1, it is also an
  eigenvector of \(\Gamma^{-1}(\beta m, t)\) for the same eigenvalue.
  Therefore, we find
  \begin{equation}
    \sum_{b=1}^3 \alpha_b = \sum_{a=1}^3 m_a = 1,
  \end{equation}
  so \(\alpha\) is an element of \(H\).  Next, we show that the
  catastrophe manifold is the graph of \(\chi\).  The differential of
  \(G_{\alpha, \beta, t}\) is given by
  \begin{equation}
    G_{\alpha,\beta, t}'(m) = \beta\left( m_a - \sum_{b=1}^3 \alpha_b \Gamma_{b,a}(\beta m, t) \right)_{a=1}^3
  \end{equation}
  Since \(\Gamma(\beta m, t)\) is invertible, the equation
  \(G_{\alpha, \beta, t}'(m) = 0\) can be solved for \(\alpha\) and we
  find \(\alpha = \chi(m, \beta, t)\).  Assume \(\alpha\) is in
  \(\Delta^2\), then \(G_{\alpha, \beta, t}'(m) = 0\) implies that
  \(m\) also lies in \(\Delta^2\) since
  \(0 < \Gamma_{b,a}(\beta m, t) < 1\) for all \(b, a\) in
  \(\{1, 2, 3\}\).

  To show \eqref{eq:hessian_matrix} and \eqref{eq:elem_sing:deg_cond}
  observe that the second derivative of \(G_{\alpha, \beta, t}\) is
  given by the matrix
  \begin{equation}
    \begin{split}
      G_{\alpha, \beta, t}''(m) &= \beta\left( 1_{a=b} - \beta
        \sum_{c=1}^3 \alpha_c \pdv{\Gamma_{c, a}}{M_b}
      \right)_{b,a=1}^3
      \\
      &= \beta\left( 1_{a=b} - \beta \sum_{c=1}^3 \alpha_c \Gamma_{c,
          a} \Big(1_{a=b} - \Gamma_{c, b} \Big)
      \right)_{b,a=1}^3
    \end{split}
  \end{equation}
  where \(\Gamma = \Gamma(\beta m, t)\).  The partial derivatives of
  \(\Gamma_{c, a}\) are elements of the tangent space of \(\Delta^2\)
  for every \(c\) in \(\{1, 2, 3\}\), that is, summing over \(a\)
  yields zero.  Therefore:
  \begin{equation}
    \left\langle h, G_{\alpha, \beta, t}''(m) \pdv{}{z}\right\rangle =
    \beta \sum_{b, a=1}^3 1_{a=b} h_b \left(\pdv{}{z}\right)_a =
    \sum_a h_a
  \end{equation}
  Since the coordinate basis of the \((x, y, z)\)-chart is an
  orthogonal basis, we find
  \(\langle \pdv{}{x}, G_{\alpha, \beta, t}''(m) \pdv{}{z}\rangle =
  \langle \pdv{}{y}, G_{\alpha, \beta, t}''(m) \pdv{}{z}\rangle = 0\)
  and
  \(\langle \pdv{}{z}, G_{\alpha, \beta, t}''(m) \pdv{}{z}\rangle =
  \frac{3}{\beta}\).  Since \(\beta > 0\) and
  \(\alpha = \chi(m, \beta, t)\), the condition for degenerate
  stationary points \(\det G_{\alpha, \beta, t}''(m) = 0\) is
  equivalent to equation~\eqref{eq:elem_sing:deg_cond}.
\end{proof}

\subsection{Universality hypothesis connecting the mid-range dynamical
  model with the static model}

In our work we are guided by the following \textit{universality
  hypothesis}, which provides a useful organizing principle to
understand the transitions which appear.  It is suggested by the
universality seen in local bifurcation theory, and verified for our
model in the full set of mid-range temperatures \(\beta < 3\), by
means of our analytical treatment in the sequel of the paper, aided in
some parts by computer algebra and numerics.

There exists a map from the two parameters temperature and time of the
dynamical model to one effective temperature parameter of the static
model of the form
\begin{equation}
  (\beta, t)\mapsto \beta_{\mathrm{st}}(\beta, t)
\end{equation}
which for our model is defined on the whole subset
$\{(\beta, t) \:|\: 0 < \beta < 3, t > 0\}$ of the positive quadrant
(and not only locally) and this map has the following property.

At fixed $(\beta, t)$ the bifurcation set slice
$B(\beta,t)\subset \Delta^2$, in the space of endconditionings $\alpha$
for the dynamical model, is diffeomorphic to a subset of the
corresponding bifurcation set slice
$B_{\mathrm{st}}(\beta_{\mathrm{st}})\subset \Delta^2$ of the
static model under a smooth $(\beta, t)$-dependent map
\begin{equation}
\Delta^2\ni\alpha\mapsto \alpha_{\mathrm{st}}(\alpha, \beta, t).
\end{equation}
See \cite[][Figure~2, page~973]{KuMe20} for nine prototypical examples
of such slices for the static model.  Moreover the corresponding
Maxwell sets of the dynamical and the static model get mapped onto each other
by the same diffeomorphism.  For corresponding values of
$(\alpha, \beta, t)$ for the dynamical model and
$(\beta_{\mathrm{st}},\alpha_{\mathrm{st}})$ the structure of
stationary points of the rate functionals of the dynamical and the
static model is identical.  The image of $\Delta^2$ under
$\alpha_{\mathrm{st}}(\cdot, \beta, t)$, which we call the
\emph{effective observation window}, always contains the uniform
distribution.  However, it may be much smaller than $\Delta^2$ for
some parameter values.  In fact, this will happen as
$t \uparrow \infty$, as we will see.  The map
$\beta_{\mathrm{st}}(\beta, t)$ from dynamical to static parameters is
(only) uniquely defined on the critical lines EW, B2B and BU of the
dynamical model (see Figure~\ref{fig:dynamical_phase_diagram}) which
get mapped to the corresponding static values
\(\beta_{\mathrm{st}}= 4\log 2\),
\(\beta_{\mathrm{st}} = \frac{8}{3}\), and
\(\beta_{\mathrm{st}} = \frac{18}{7}\)
\parencite[see][Table~1]{KuMe20}.

The following conjecture underlies this hypothesis, as it expresses
the structural similarity of dynamical and static rate functional, by
means of a parameter-dependent map acting on the state space
$\Delta^2$, compare with the definition of equivalent potentials in
\cite[][Chapter~6, Section~1]{PoSt78}.

\begin{conjecture}
  There exists a set \(U\) which contains the unit simplex
  \(\Delta^2\) and is open in the hyperplane \(H\) such that
  \begin{enumerate}
  \item there exists a smooth map \(\psi_1\) from the subset
    \begin{equation}
      D = \{ (\alpha, \beta, t) \:|\: \beta < 3, t > 0, \alpha \in U\}
    \end{equation}
    of the parameter space of the time-evolved model to the parameter
    space \((0, \infty)\times\Delta^2\) of the static model such that
    the map
    \((\alpha, \beta) \mapsto \psi_1(\alpha, \beta, t_0)\) is a
    diffeomorphism from \(D \cap \{(\alpha, \beta, t) : t = t_0\}\) to
    the respective image of this intersection for every \(t_0>0\).
  \item there exists a smooth map \(\psi_2\) from \(D\times \Delta^2\)
    to the state space \(\Delta^2\) of the static model such that
    the map
    \(m \mapsto \psi_2(\alpha, \beta, t, m)\) defined on the interior
    of \(\Delta^2\) is a diffeomorphism onto its image  for
    every \((\alpha, \beta, t)\) in \(D\).
  \item For every \((\alpha, \beta, t)\) in \(D\) and every \(m\) in
    \(\Delta^2\) the following identity holds:
    \begin{equation}
      G_{\alpha, \beta, t}(m) = f_{\psi_1(\alpha, \beta, t)}\circ \psi_2(\alpha, \beta, t, m)
    \end{equation}
    where \(f_{\beta, \alpha}\) denotes the potential (5) of the
    static model \parencite[see][Sect. 1.2]{KuMe20}.
  \item There exists a function
    \((\beta, t) \mapsto \beta_{\mathrm{st}}(\beta, t)\) on
    \((0, 3) \times (0, \infty)\) such that
    \begin{equation}
      \mathrm{pr}_1\circ\psi_1(\alpha, \beta, t) = \beta_{\mathrm{st}}(\beta, t)
    \end{equation}
    where \(\mathrm{pr}_1\) denotes the projection
    \((0, \infty)\times\Delta^2 \to (0, \infty)\).  In other words,
    the effective static inverse temperature
    \(\beta_{\mathrm{st}}\) does not depend on the dynamical
    \(\alpha\).
  \end{enumerate}
\end{conjecture}

A comparison of Figure~\ref{fig:tpe:contours} with
\cite[][Figure~5]{KuMe20} gives evidence for the existence of the
map \(\psi_1\) as the bifurcation set slice of the static model looks
structurally similar to the bifurcation set slice in a neighbourhood
of the unit simplex of the dynamical model.  The contour plots in the
rightmost plots of the two figures support the existence of the map
\(\psi_2\) as the contour plot of the dynamical potential
\(G_{\alpha, \beta, t}\) looks structurally similar to a subset of the
contour plot of the static
potential~\(f_{\beta_{\mathrm{st}}(\beta, t),
  \alpha_{\mathrm{st}}(\alpha, \beta, t)}\).  Note, however, that we
are not going to construct the maps \(\psi_1\) and \(\psi_2\) in the
following sections of the paper and we do not need to do it.  Instead,
we explicitly compute the critical lines from the dynamical potential
following the ideas of singularity theory.  This means that the lines
can be found independently of the construction of the maps \(\psi_1\)
and \(\psi_2\).  The behavior of the model in the vicinity of these
lines follows from Thom's classification theorem
\parencite[see][Section~5 of Chapter~3]{Lu76} and our global analysis
is supported by the global numerical analysis of the relevant parts of
the dynamical bifurcation set.  In the following sections we now
proceed with the discussion of the critical lines.

\subsection{The symmetric cusp exit (SCE) line and the non-Gibbs
  temperature}

The non-Gibbs inverse temperature \(\beta_{\mathrm{NG}}\) is defined
as the supremum of all \(\beta\) such that \(\mu_{n, \beta, t}\) is
sequentially Gibbsian for all positive \(t\).  It turns out to be a
maximum.  As the type of transitions of the dynamical model for
mid-range temperatures can be understood in terms of the static case,
let us remark that in the static Potts model the first type of bad
magnetic fields that show up with increasing \(\beta\) are due to
three symmetric cusp singularities, the \enquote{rockets}
\parencite[see][Figures~2 and~4]{KuMe20} and that there are
no bad magnetic fields for \(\beta \le 2\).  Therefore, in the
dynamical model, we look for symmetric cusp points that have just
passed the simplex edges in their midpoint and moved outside, which
leads us to the \emph{symmetric cusp exit line} in the dynamical phase
diagram.  Without loss of generality by symmetry we consider the
simplex edge where \(\alpha_1 = 0\).

\begin{proposition}
  \label{prop:sce}
  Fix any positive \(\beta\) and \(t\), let \(m\) be a point in \(H\)
  with \((x, y, z)\)-coordinates \((0, y, 0)\).
  \begin{enumerate}
\item The point \(\alpha = \chi(m, \beta, t)\) in \(H\) is a symmetric
        cusp point on the simplex edge if and only if
  \begin{align}
    \frac{6}{\beta} y + \frac{e^{g_t} + 1 - 2e^{3y}}{e^{g_t} + 1 + e^{3y}} &= 0 \label{eq:sce:mf}\\
    \frac{6}{\beta} + \frac{3(e^{g_t} - 1)^2}{(e^{g_t} + 1 + e^{3y})^2} - \frac{3(e^{g_t} + 1)}{e^{g_t} + 1 + e^{3y}} &= 0. \label{eq:sce:deg}
  \end{align}

\item The solutions of the system (\ref{eq:sce:mf}--\ref{eq:sce:deg})
        can be explicitly parametrized in the form
  \begin{align}
    \beta &= \frac{2s (2e^s + F(s))}{4e^s - F(s)} \label{eq:sce:beta}\\
    g_t &= \log(\frac{1}{2} F(s) - 1) \label{eq:sce:g_t}
  \end{align}
  where
  \begin{equation}
    \label{eq:sce:helper}
    F(s) = -(s-1)e^s - 4s + \sqrt{
      ((s - 1)e^s + 4s)^2 + 8(2s + e^{2s})}.
  \end{equation}
  for \(s < 0\).
\item The non-Gibbs temperature is given via
  \begin{equation}
    \beta_{\mathrm{NG}} = \frac{2s_0 (2e^{s_0} + F(s_0))}{4e^{s_0} - F(s_0)} \approx 2.52885
  \end{equation}
  where \(s_0\) is the unique zero in \((-\infty, 0)\) of
  \begin{equation}
    \label{eq:ng:zero-problem}
    \begin{split}
      \frac{
        64 \, s^{3} + 64 \, s^{2}
        + s(s^2 + s + 6) e^{3 \, s}
        + 4 \, s(5 \, s + 6) e^{2 \, s}
        - 8 \, s(2 \, s - 3) e^{s}
      }{
        \sqrt{
          ((s - 1)e^s + 4s)^2 + 8(2s + e^{2s})
        }
      }
      \\
      - 16 \, s^{2} - s(s + 2) e^{2 \, s}
      + 4 \, s(s - 2) e^{s} - 8 \, s.
    \end{split}
  \end{equation}
\end{enumerate}

\end{proposition}

\begin{proof}
  Let us first prove item 1.  A symmetric cusp point \(\alpha\) is the
  image of a symmetric degenerate stationary point \(m\) under the map
  \(\chi(\cdot, \beta, t)\) at which the tangent vector of the curve
  of degenerate stationary points (given by vanishing Hessian
  determinant) is parallel to the direction of degeneracy.  The
  partial derivatives of \(G_{\alpha, \beta, t}\) with respect to
  \(x\) and \(z\) vanish at \(m\) because of symmetry, so it is
  sufficient for a stationary point \(m\) to have a vanishing partial
  derivative with respect to the \(y\)-coordinate of \(m\).  Now, for
  the gradient we note that
  \begin{equation}
    \begin{split}
      \pdv{G_{\alpha, \beta, t}}{y} &= 2m_1 - m_2 - m_3 - \sum_{b=1}^3
      \alpha_b (2\Gamma_{b, 1} - \Gamma_{b, 2} - \Gamma_{b, 3})
      \\
      &= \frac{6}{\beta} y - \sum_{b=1}^3 \alpha_b (3\Gamma_{b,1} - 1)
      \\
      &= \frac{6}{\beta} y + 1 - \frac{3e^{3y}}{e^{3y} + e^{g_t} + 1}
    \end{split}
  \end{equation}
  where we have abbreviated
  \(\Gamma_{b,a} = \Gamma_{b,a}(\beta m, t)\) and used the fact that
  \(\alpha\) lies on the simplex edge
  \(\alpha = (0, \frac{1}{2}, \frac{1}{2})\).  This yields
  Equation~\eqref{eq:sce:mf}.

  We will now derive equation~\eqref{eq:sce:deg}.  Note that the
  mixed partial derivative, which appears in the degeneracy
  condition~\eqref{eq:elem_sing:deg_cond}, vanishes at partially
  symmetric points:
  \begin{equation}
    \label{eq:sce:mixed_partial}
    \pdv[2]{G_{\alpha, \beta, t}}{x}{y} = -3 \sum_{b=1}^3 \alpha_b
    \pdv{\Gamma_{b, 1}}{x} = 3\sqrt{3} \sum_{b=1}^3 \alpha_b
    \Gamma_{b, 1}(\Gamma_{b,3} - \Gamma_{b, 2})
  \end{equation}
  Plugging in \(\alpha = (0, \frac{1}{2}, \frac{1}{2})\), the
  right-hand side of the last equality in \eqref{eq:sce:mixed_partial}
  vanishes because
  \(\Gamma_{3, 3} - \Gamma_{3, 2} = \Gamma_{2, 2} - \Gamma_{2, 3}\)
  for points \(m\) which have the partial symmetry \(m_2 = m_3\).
  Therefore the degeneracy condition~\eqref{eq:elem_sing:deg_cond} is
  in product form.  We calculate the remaining partial derivatives:
  \begin{align}
    \label{eq:sce:Gyy}
    \pdv[2]{G_{\alpha, \beta, t}}{y} &= \frac{6}{\beta} - 9(\Gamma_{2,1} - \Gamma_{2,1}^2)
                                       = 9\left(\Gamma_{2,1} - \frac{1}{2}\right)^2 - \frac{9}{4} + \frac{6}{\beta}
    \\
    \label{eq:sce:Gxx}
    \pdv[2]{G_{\alpha, \beta, t}}{x} &=
      \frac{6}{\beta} - 3\Big(\Gamma_{2,2} + \Gamma_{2, 3} - (\Gamma_{2, 3} - \Gamma_{2, 2})^2\Big)
  \end{align}
  The partial derivative \eqref{eq:sce:Gyy} is always positive for
  \(\beta < \frac{8}{3}\).  This means we only have to consider the
  zeros of \eqref{eq:sce:Gxx}.  This yields
  equation~\eqref{eq:sce:deg}.

  We will now explain the parametrization of the set of solutions
  given in 2.  First note that the variable \(\beta\) can be
  eliminated from Equation~(\ref{eq:sce:deg}) using
  Equation~(\ref{eq:sce:mf}) for all \(y \neq 0\).  When we set
  \(w = e^{g_t} + 1\) we find that the resulting equation is a
  quotient of quadratic polynomials in \(w\):
  \begin{align}
    -\frac{
    w^2 + ((3y - 1)e^{3y} + 12y)w - 2(6y + e^{6y})
    }{
    y(w + e^{3y})^2
    } = 0
  \end{align}
  Since \(w > 2\), it suffices to consider the numerator of the
  left-hand side.  The discriminant of this quadratic polynomial is
  given by
  \begin{equation}
    D = ((3y - 1)e^{3y} + 12y)^2 + 8(6y + e^{6y}).
  \end{equation}
  It is positive for all real \(y\).  Therefore, this polynomial has
  two real roots.  Because \(w > 2\), we choose the larger of the two
  solutions
  \begin{equation}
    w = \frac{1}{2}\left( -(3y - 1)e^{3y} - 12y + \sqrt{D} \right) = \frac{1}{2} F(s)
  \end{equation}
  where we have defined \(s = 3y\) and used the definition of \(F(s)\)
  in Equation~(\ref{eq:sce:helper}).  Furthermore, \(F(s) > 4\) for
  \(s\neq 0\) such that Equation~(\ref{eq:sce:g_t}) yields positive
  values for \(g_t\).

  Finally, the non-Gibbs inverse temperature is the minimal value of
  \(\beta\) along the curve given by the
  parametrization~(\ref{eq:sce:beta}--\ref{eq:sce:g_t}).  Therefore we
  calculate the derivative of \eqref{eq:sce:beta} which gives
  \begin{equation}
    \dv{\beta}{s}=
    -2\cdot\frac{2(3s-1)e^sF(s) - 6se^sF'(s) + F^2(s) - 8e^{2s}}{(4e^s - F(s))^2}.
  \end{equation}
  Since \(4e^s - F(s)\) is never zero for any \(s\) in
  \((-\infty, 0)\), we only have to consider the numerator of the
  fraction.  We calculate the derivative of \(F\)
  \begin{equation}
    F'(s) = -se^s - 4 + \frac{((s-1)e^s + 4s)(4 + se^s) + 8(1 + e^{2s})
    }{
      \sqrt{((s - 1)e^s + 4s)^2 + 8(2s + e^{2s})}}.
  \end{equation}
  Plugging everything together, \(\mathrm{d}\beta/\mathrm{d}s = 0\) is
  exactly fulfilled for the zero of the function defined in
  (\ref{eq:ng:zero-problem}).
\end{proof}

\begin{lemma}
  Suppose \(\beta\) lies in the interval \((\beta_{\mathrm{NG}}, 3)\).
  The entry time \(t_{\mathrm{NG}}(\beta)\) into the non-Gibbs region
  is given by
  \begin{equation}
    \label{eq:t_NG}
    t_{\mathrm{NG}}(\beta) = \frac{1}{3} \log
    \frac{2(\beta - 3y)e^{3y} + \beta + 6y}{2((\beta - 3y)e^{3y} - \beta - 6y)}
  \end{equation}
  where \(y\) is the largest root in \((-\frac{\beta}{6}, 0)\) of
  \begin{equation}
    \begin{split}
      \label{eq:t_NG:zero-problem}
      y \mapsto &2 \, \beta^{2} + 24 \, \beta y + 72 \, y^{2} - \left(\beta^{2} + 3 \, \beta y - 18 \, y^{2} - 9 \, \beta\right) e^{6 \, y} \\&- 4 \, {\left(\beta^{2} + 3 \, \beta y - 18 \, y^{2}\right)} e^{3 \, y}.
    \end{split}
  \end{equation}
\end{lemma}

\begin{proof}
  The entry time \(t_{\mathrm{NG}}\) is given by the first entry of
  \emph{rockets} into the unit simplex while increasing the time \(t\)
  and keeping \(\beta\) fixed.  This is because, if the pentagrams
  unfold at all under increase of time, they unfold after the rockets
  have entered the unit simplex~\(\Delta^2\).  This will be clear in
  the next subsection where we compute the butterfly line. So let us
  consider the system~(\ref{eq:sce:mf}--\ref{eq:sce:deg}) and fix any
  positive \(\beta < 3\).  Since the relation~(\ref{eq:def:g_t})
  between \(g_t\) and \(t\) is strictly monotonically decreasing, we
  have to look for the maximal \(g_t\) such that \((\beta, g_t, y)\)
  with negative \(y\) is a solution to the
  system~(\ref{eq:sce:mf}--\ref{eq:sce:deg}), which defines the
  symmetric cusp exit line.  Here, \(y\) is a magnetization-type
  variable.  We can solve Equation~\eqref{eq:sce:mf} for
  \(w = e^{g_t} + 1\) to obtain
  \begin{equation}
    \label{eq:t_NG:w}
    w = \frac{2(\beta - 3y)e^{3y}}{\beta + 6y}.
  \end{equation}
  Plugging this into the left-hand side of the degeneracy
  condition~\eqref{eq:sce:deg}, we arrive at
  \begin{equation}
    \begin{split}
      \frac{2e^{-6y}}{3\beta^2}
      \Big( 2\beta^2 + 24\beta y + 72y^2 &- (\beta^2 +3\beta y -18y^2 - 9\beta)e^{6y}
      \\
      & -4(\beta^2 + 3\beta y - 18 y^2)e^{3y} \Big) = 0.
    \end{split}
  \end{equation}
  This yields the expression of \eqref{eq:t_NG:zero-problem}.  Since
  the right-hand side of \eqref{eq:t_NG:w} is increasing with \(y\),
  we have to pick the largest root of \eqref{eq:t_NG:zero-problem}.
\end{proof}

\begin{figure}
  \centering
  \includegraphics{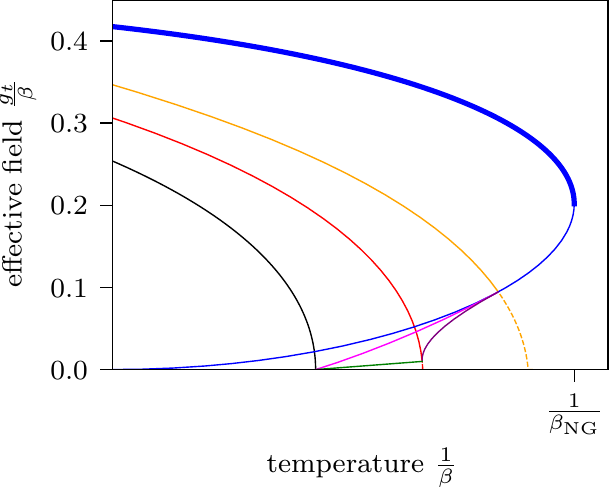}
  \caption{The thick blue line, which ends at the non-Gibbs
    temperature~\(\frac{1}{\beta_{\mathrm{NG}}}\), marks the entry
    time in the dynamical phase diagram.  Time is a monotonically
    decreasing function of \(g_t\) so the first time we hit the
    symmetric cusp exit line when moving on a vertical line of fixed
    temperature corresponds to the entry time.}
  \label{fig:entry_time}
\end{figure}
\FloatBarrier

\subsection{The butterfly unfolding (BU) line and butterfly exit
  temperature}
\label{sec:bu}

The unfolding of the pentagrams is a very important mechanism since it
changes the set of bad empirical measures from straight lines to
Y-shaped, branching curves.  This mechanism is already present in the
static case, however, in contrast to the static case we have to deal
with the fact that in some parameter regions the pentagrams do not
fully lie inside of the unit simplex.  This leads us to the definition
of a \emph{butterfly exit inverse temperature}~\(\beta_{\mathrm{BE}}\)
for which at some point in time \(t > 0\) there is a cusp point on an
edge of the simplex that is about to unfold into a pentagram.  By
definition, \(\beta_{\mathrm{BE}}\) lies between
\(\beta_{\mathrm{NG}}\) and \(\frac{8}{3}\).  The value
\(\frac{8}{3}\) is the first inverse temperature for which a
beak-to-beak scenario inside of the unit simplex appears as we will
see in Section~\ref{sec:b2b_line}.

\begin{proposition}
  \label{prop:be}
  Let \(v(m, \beta, t) = (\varphi_{\beta})_2 \circ \chi(m, \beta, t)\)
  be the parallel coordinate of \(\chi(m, \beta, t)\) and let
  \(\beta(s)\) and \(t(s)\) be given by
  \emph{(\ref{eq:sce:beta}--\ref{eq:sce:g_t})}.  The butterfly exit
  \(\beta_{\mathrm{BE}}\) is given by
  \begin{equation}
    \beta_{\mathrm{BE}} =
    \frac{2s_0 (2e^{s_0} + F(s_0))}{4e^{s_0} - F(s_0)} \approx 2.59590
  \end{equation}
  where \(s_0 < 0\) is the largest zero of
  \begin{equation}
    \label{eq:be:zero_problem}
    s \mapsto \pdv[2]{v}{x}\Big(m(s), \beta(s), t(s)\Big) +
    \pdv{v}{y}\Big(m(s), \beta(s), t(s)\Big) \ddot\gamma_s(0)
  \end{equation}
  and \(\gamma_s\) is the implicit function \(y = \gamma_s(x)\)
  defined in a neighbourhood of \((x, y) = (0, \frac{s}{3})\) by the
  degeneracy condition \eqref{eq:elem_sing:deg_cond}.
\end{proposition}

Note that equation~\eqref{eq:be:zero_problem} is explicitly computed
by a computer program because its expression is very complicated.
Nevertheless it is possible to plot the function (see
Figure~\ref{fig:be:zero_problem_plot}).

\begin{figure}
  \centering
  \includegraphics{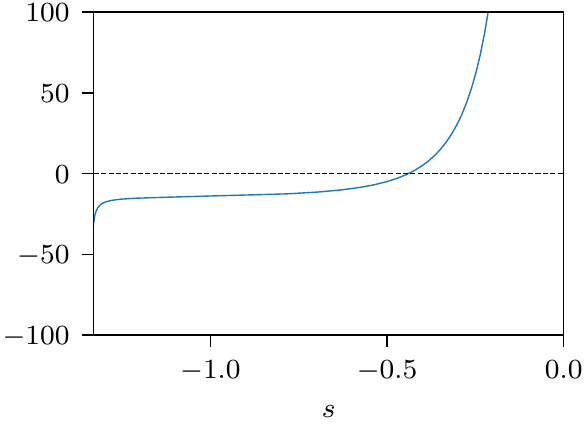}
  \caption{This figure shows a plot of the function
    \eqref{eq:be:zero_problem} which is involved in the expression for
    the \emph{butterfly exit (BE) temperature} in
    Proposition~\ref{prop:be}.}
  \label{fig:be:zero_problem_plot}
\end{figure}

\begin{proof}
  Let us first fix \(\beta\) between \(\beta_{\mathrm{NG}}\) and
  \(\frac{8}{3}\) and a positive \(t\).  Consider a point \(\alpha\)
  on the midpoint of one of the edges of \(\Delta^2\) such that
  \((\alpha, \beta, t)\) belongs to the bifurcation set.  Furthermore,
  without loss of generality by symmetry let us assume that
  \(\alpha_2 = \alpha_3\).  To this point corresponds a degenerate
  stationary point \(m\) that has the same symmetry \(m_2 = m_3\).  We
  can solve the degeneracy condition~\eqref{eq:elem_sing:deg_cond} in
  a neighbourhood of \(m\) in the form \(y = \gamma_{\beta, t}(x)\)
  such that \(\gamma_{\beta, t}(0)\) is the \(y\)-coordinate of \(m\).
  In \(\alpha\)-space in a neighbourhood of
  \(\alpha = \chi(m, \beta, t)\) we can now write the bifurcation set
  as
  \(\chi(\varphi_{\beta}^{-1}(x, \gamma_{\beta, t}(x), 0), \beta,
  t)\).  We know that the parallel component \(v\) of \(\alpha\)
  fulfills
  \begin{equation}
    \dv[2]{}{x} v\qty(\gamma_{s_0}(x), \beta_\ast, t_\ast)  = 0
  \end{equation}
  when we follow the curve \(\gamma_s\) through the bifurcation set.
  This is because it has a minimum before the pentagram unfolds and it
  has a maximum after the pentagram has unfolded.  The curve
  \(\gamma\) of degenerate stationary points is obtained by solving
  equation~\eqref{eq:elem_sing:deg_cond} in the form \(y = \gamma(x)\)
  around \((0, y^\ast)\) where \(y^\ast\) is the parallel component of
  \(m^\ast\).  Let us now compute the second derivative of the
  \(v\)-component of the curve:
\begin{equation}
  \begin{split}
    \dv[2]{v}{x}\bigg|_{x=0} &= \dv{x}\left( \pdv{v}{x} + \pdv{v}{y} \dot\gamma(x) \right)    \\
    &= \pdv[2]{v}{x}\bigg|_{x=0} + \pdv{v}{y}\bigg|_{x=0} \ddot\gamma(0)
  \end{split}
\end{equation}
The other mixed partial derivatives of \(v\) vanish since
\(\dot\gamma(0) = 0\) because of symmetry.
\end{proof}

Furthermore, we compute \(\ddot\gamma(0)\) via implicit
differentiation: Let us write \(f(x, y)\) for the left-hand side of
\eqref{eq:elem_sing:deg_cond} viewed as a function in the unit simplex
in \((x, y)\)-coordinates.  By implicit differentiation we then find:
\begin{equation}
  \dot\gamma(x) = -\pdv{f}{x} \bigg/ \pdv{f}{y}
\end{equation}
And therefore:
\begin{equation}
  \begin{split}
    \ddot\gamma(0) &= -\frac{\pdv[2]{f}{x}}{\pdv{f}{y}} +
    \frac{\pdv{f}{x} \pdv[2]{f}{x}{y}}{\qty(\pdv{f}{y})^2}
    = -\frac{\pdv[2]{f}{x}}{\pdv{f}{y}} - \dot\gamma(0)
    \frac{\pdv[2]{f}{x}{y}}{\pdv{f}{y}}
      \\
      &= -\pdv[2]{f}{x}\bigg/ \pdv{f}{y}.
  \end{split}
\end{equation}
Using the symbolic calculus tools (see
page~\pageref{source_code_info}) we can obtain an expression for
\eqref{eq:be:zero_problem}.

\begin{figure}
  \centering
  \includegraphics{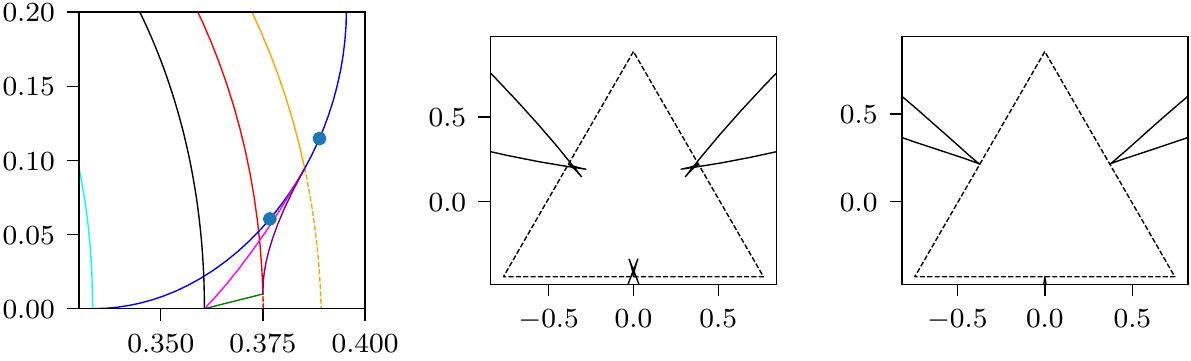}
  \caption{This figure shows the bifurcation set sliced at two points
    on the symmetric cusp exit line.  The left plot shows a slice
    before the butterfly exit point is passed (lower point in the
    phase diagram), the right plot shows a slice after the
    butterfly exit point (intersection point of yellow and blue line)
    on the symmetric cusp exit line.}
  \label{fig:butterfly_exit_sce}
\end{figure}

Using a similar approach it is possible to compute the line in the
dynamical phase diagram for which we find butterfly points no matter
where these points are with respect to the unit simplex.  The key idea
that the parallel component of the curve in \(\alpha\)-space has a
vanishing second derivative with respect to the curve parameter stays
the same.  But since we do not restrict the point in \(\alpha\)-space
to lie on the unit simplex we lose one equation and we end up with a
one-dimensional set of solutions.

\begin{proposition}
  For \(\beta\) in \((\beta_{\mathrm{BE}}, \frac{8}{3})\) the
  butterfly unfolding happens at the unique \emph{butterfly transition
    time}~\(t_{\mathrm{BU}}(\beta)\) which is obtained as follows:
  Define a function \(H\) via
  \begin{equation}
    \label{eq:bl:w}
    H(\beta, s) = H_1(\beta, s) + \sqrt{H_2(\beta, s)}
  \end{equation}
  where
  \begin{align*}
    H_1(\beta, s) &= \beta e^{2s} - se^{2s} + 4\beta e^s - 4se^s + \beta + 2s - 3e^{2s} - 3e^s\\
    H_2(\beta, s) &= {\left(\beta^{2} - 2 \, {\left(\beta - 3\right)} s + s^{2} - 6 \, \beta + 9\right)} e^{4 \, s}
    \\
                  &\quad + 2 \, {\left(4 \, \beta^{2} - {\left(8 \, \beta - 9\right)} s + 4 \, s^{2} - 9 \, \beta - 9\right)} e^{3 \, s}\\
                  &\quad + 3 \, {\left(6 \, \beta^{2} - 2 \, {\left(5 \, \beta - 6\right)} s + 4 \, s^{2} - 18 \, \beta + 3\right)} e^{2 \, s}
    \\
                  &\quad + 2 \, {\left(4 \, \beta^{2} + 2 \, {\left(2 \, \beta - 15\right)} s - 8 \, s^{2} - 15 \, \beta\right)} e^{s}
    \\
    &\quad + \beta^{2} + 4 \, \beta s + 4 \, s^{2}
  \end{align*}
  and a function
  \begin{equation}
    t(\beta, s) = \frac{1}{3} \log \frac{
      H(\beta, s) + 6e^s
    }{
      H(\beta, s) - 12e^s
    }.
  \end{equation}
  Then the butterfly transition time \(t_{\mathrm{BU}}(\beta)\) is
  given by
  \begin{equation}
    \label{eq:bu:t}
    t_{\mathrm{BU}}(\beta) = t(\beta, s_\ast(\beta)) =
    \frac{1}{3} \log \frac{
      H(\beta, s_\ast(\beta)) + 6e^{s_\ast(\beta)}
    }{
      H(\beta, s_\ast(\beta)) - 12e^{s_\ast(\beta)}
    }
  \end{equation}
  and \(s_\ast(\beta) < 0\) is the largest zero of
  \begin{equation}
    \label{eq:bl:zero_problem}
    s \mapsto \pdv[2]{v}{x}\Big(\varphi_{\beta}\big(0, \frac{s}{3}, 0\big), \beta, t(\beta, s)\Big) +
    \pdv{v}{y}\Big(\varphi_{\beta}\big(0, \frac{s}{3}, 0\big), \beta, t(\beta, s)\Big) \ddot\gamma_{\beta, t(\beta, s)}(0).
  \end{equation}
\end{proposition}

\begin{proof}
  Using the same reasoning as in the proof of
  Proposition~\ref{prop:be}, we find that the point \(m\) maps under
  \(\chi(\cdot, \beta, t)\) to a point \(\alpha\) that is about to
  unfold into a pentagram if
  \begin{equation}
    \label{eq:bl:minmax}
    \dv[2]{}{x}\Bigg|_{x=0} v(\varphi_{\beta}^{-1}(x, \gamma_{\beta, t}(x), 0), \beta, t) = 0
  \end{equation}
  where \(\gamma_{\beta, t}\) is obtained by solving the degeneracy
  condition~\eqref{eq:elem_sing:deg_cond} in the form
  \(y = \gamma_{\beta, t}(x)\) in a neighbourhood of the point \(m\).
  This equation is now dependent on \(m, \beta\) and \(t\), that is,
  we have one equation and three variables (\(m\) is one-dimensional
  because \(m_2 = m_3\)).  Additionally, since we know that the
  direction of degeneracy is the \(x\)-direction, we have the equation
  \begin{equation}
    \pdv[2]{G_{\alpha, \beta, t}}{x}\Bigg|_{x=0} = 0.
  \end{equation}
  This equation can be solved for \(w = e^{g_t} + 1\) which yields
  \eqref{eq:bl:w}.  Plugging this into \eqref{eq:bl:minmax}, we are
  left to find the zeros of \eqref{eq:bl:zero_problem} for some fixed
  \(\beta\) in the interval~\((\beta_{\mathrm{BE}}, \frac{8}{3})\).
\end{proof}

\subsection{Reentry into Gibbs: the asymmetric cusp exit (ACE) line}
\label{sec:ace}

\begin{figure}
  \centering
  \includegraphics{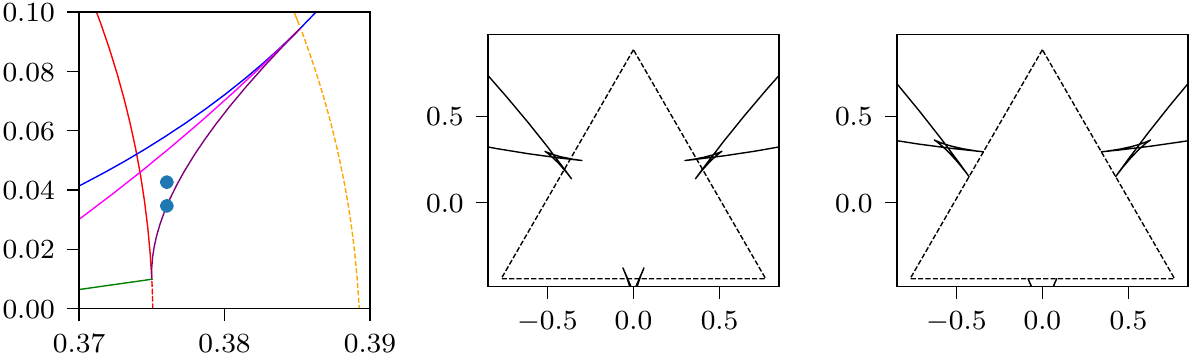}
  \caption{ This figure shows two bifurcation set slices that
    illustrate the exit of the asymmetric cusp points.  The central
    plot shows the bifurcation set slice for a time at which the exit
    has not yet happened (upper point in the phase diagram).  The
    rightmost plot shows the bifurcation set slice exactly on the
    purple line ACE, that is, when the exit is just happening.}
  \label{fig:ace:bifu_slices}
\end{figure}

In the \(\beta\)-regime
\((\beta_{\mathrm{NG}}, \beta_{\mathrm{BE}})\), three pentagrams
unfold inside of the simplex at an intermediate time and leave the
simplex as \(t\) increases further.  Since we are interested in
phase-coexistence of the first layer model \(\bar\mu_n\)
(Lemma~\ref{lem:cond_first_layer}) and the phase-coexistence lines of
the pentagram end in the asymmetric cusp points of the pentagrams, we
must compute the exit time \(t_{\mathrm{G}}(\beta)\) of these points
for \(\beta\) in the above regime.  Like in the previous subsection,
this is done using a combination of symbolic and numerical computation
(see page~\pageref{source_code_info}).  First, let us state the
problem that we need to solve.

\begin{proposition}
  Fix a positive \(\beta\) and positive \(t\) and consider the set of
  solutions \(m\) to the degeneracy
  condition~\eqref{eq:elem_sing:deg_cond} with
  \(\alpha = \chi(m, \beta, t)\).

  \begin{enumerate}
  \item There is exactly one branch of solution with \(m_2 = m_3\) and
    it is given by the graph of a map
    \(x \mapsto y = \gamma_{\beta, t}(x)\).
  \item Furthermore, define the map \((x, y) \mapsto v(x, y)\) via
  \begin{equation}
    v(x, y) = (\varphi_{\beta})_1 \circ \chi( \varphi_{\beta}^{-1}(x, y, 0), \beta, t).
  \end{equation}
  Then the asymmetric cusps of the pentagrams are on the simplex edges
  if and only if
  \begin{align}
    \label{eq:asym_cusps:simpl_egde}
    v\qty(x, \gamma_{\beta, t}(x)) &= -\frac{1}{6}\beta\\
    \label{eq:asym_cusps:max}
    \pdv{v}{x} \qty(\varphi_{\beta}(x, \gamma_{\beta, t}(x), 0))
    + \pdv{v}{y} \qty(\varphi_{\beta}(x, \gamma_{\beta, t}(x), 0)) \dot\gamma_{\beta, t}(x) &= 0.
  \end{align}
\end{enumerate}
\end{proposition}

\begin{proof}
  The location of the asymmetric cusps of the pentagrams on the curve
  \(x \mapsto \chi(\varphi^{-1}_\beta(x, \gamma_{\beta, t}, 0), \beta,
  t)\) are given by the local maxima of the parallel component
  \(v(x)\) as a function of the curve parameter \(x\) (see
  Figure~\ref{fig:ace:bifu_slices}).  This yields
  \eqref{eq:asym_cusps:max}.
  Equation~\eqref{eq:asym_cusps:simpl_egde} comes from the constraint
  that the cusp point lies on the simplex edge because for points on
  the edge the parallel component equals \(-\frac{1}{6}\beta\) in the
  chart~\eqref{eq:coordinates}.
\end{proof}

Now, similarly to the case for the butterfly line, the computation of
\(\dot\gamma_{\beta, t}(x)\) by hand is impractical.  Therefore we
compute the expression symbolically with the help of the computer.
This allows us to numerically determine the course of the line in the
dynamical phase diagram.  Now, because it is impossible to solve the
degeneracy equation~\eqref{eq:elem_sing:deg_cond} in the form
\(y = \gamma_{\beta, t}(x)\) explicitly, we proceed as follows.  Note
that it is possible to solve \eqref{eq:asym_cusps:simpl_egde} for
\(\beta\) and plug it into equation~\eqref{eq:asym_cusps:max}.  We
then fix some value of \(g_t\), and numerically solve the system
consisting of the degeneracy condition~\eqref{eq:elem_sing:deg_cond},
where \(\beta\) is substituted from \eqref{eq:asym_cusps:simpl_egde},
and equation~\eqref{eq:asym_cusps:max}, where \(\gamma_{\beta, t}\) is
substituted by \(y\) and
\begin{equation}
  \dot\gamma_{\beta, t}(x) = -\pdv{f}{x} \bigg/ \pdv{f}{y}
\end{equation}
where \(f\) denotes the left-hand side of
\eqref{eq:elem_sing:deg_cond} considered as a function of \((x, y)\).
This yields two equations in the two variables \(x\) and \(y\).

\subsection{The triple point exit (TPE) line}
\label{sec:tpe}

To each of the three pentagrams there belongs a special point, the
\emph{triple point} \parencite[see][Sects.~3.2]{KuMe20}.  This
point is characterized by the coexistence of three global minima, that
is, the functional values of all the three minimizers are equal.
First, we discuss the existence of these points and then we determine
for each fixed positive \(\beta\) the exit time
\(t_\mathrm{triple}(\beta)\).  This is the last time for which there
are bad empirical measures with partial symmetry that lie inside the
unit simplex.

\begin{proposition}
  For each pair \((\beta, t)\) in
  \begin{equation}
    \{ (\beta, t) \:|\: \beta_{\mathrm{BE}} < \beta < 4\log 2, t > t_{\mathrm{BU}}(\beta) \}
  \end{equation}
  there exists exactly one \(\alpha\) in the hyperplane \(H\) with
  \(\alpha_1 \le \alpha_2 \le \alpha_3\) such that
  \(G_{\alpha, \beta, t}\) has precisely three global minimizers.
\end{proposition}

\begin{proof}
  By symmetry, the triple point \(\alpha\) has the partial symmetry
  \(\alpha_2 = \alpha_3\).  Therefore consider the curve
  \(v \mapsto \alpha(v) = \varphi_{\beta}^{-1}(0, v, 0)\) which
  crosses the \(\alpha\)-region for which the potential
  \(G_{\alpha, \beta, t}\) has three minimizers two of which lie
  inside the same fundamental cell \(m_1 \le m_2 \le m_3\).  There is
  always such a region because the pentagrams have already unfolded
  (\(t > t_{\mathrm{but}}\)).  This gives rise to the two maps
  \(v \mapsto m(v)\) and \(v \mapsto m'(v)\) which map \(v\) to one of
  the two minimizers \(m(v)\) or \(m'(v)\) inside this cell. Assume
  that \(\varphi_\beta(m(v)) = (x(v), y(v), 0)\) and
  \(\varphi_{\beta}(m'(v)) = (0, y'(v), 0)\) with \(y'(v) > y(v)\) and
  \(x(v) > 0\).  Now, we can define the difference
  \begin{equation}
    g(v) := G_{\alpha(v), \beta, t}(m(v)) - G_{\alpha(v), \beta, t}(m'(v))
  \end{equation}
  for all \(v\) such that \(\alpha(v)\) lies in the former regime.  Therefore
  \begin{equation}
    \begin{split}
      g'(v) &= \pdv{G_{\alpha(v), \beta, t}}{v} - \pdv{G_{\alpha(v),
          \beta, t}}{v}
      \\
      &= \log \frac{(e^{g_t+2x} + e^{3y+x} + 1)(e^{g_t + 3y'} +
        2)^2(e^{g_t} + e^{2x} + e^{3y + x})}{(e^{g_t+x+3y} + e^{2x} +
        1)^2(e^{g_t} + e^{3y'} + 1)^2}
    \end{split}
  \end{equation}
  since \(m(v)\) and \(m'(v)\) are stationary points.
\end{proof}

Since the pentagrams in the bifurcation slices leave the simplex
(observation window), it is necessary for a discussion of the bad
empirical measures that we find the time when the triple points leave
the unit simplex.  The problem that we have to solve is stated in the
following proposition.

\begin{proposition}
  Fix any positive \(\beta\) in the interval
  \((\beta_{\mathrm{BE}}, 4\log 2)\) and let \(\alpha\) be the
  midpoint of the edge of the simplex with \(\alpha_2 = \alpha_3\).
  First, define the function
  \begin{equation}
    \label{eq:tpe:time_fct}
    t(\beta, y) = \frac{1}{3}
    \log(\frac{
      2(\beta - 3y)e^{3y} + \beta + 6y
    }{
      2((\beta - 3y)e^{3y} - \beta - 6y
    })
  \end{equation}
  The exit time~\(t_{\mathrm{TPE}}(\beta)\) is then given by
  \(t_{\mathrm{TPE}}(\beta) = t(\beta, y'(\beta))\) where
  \(\varphi_{\beta}(0, y'(\beta))\) and
  \(\varphi_{\beta}(x(\beta), y(\beta))\) lie in the fundamental cell
  \(m_1 \le m_2 \le m_3\) and the triple
  \((y'(\beta), x(\beta), y(\beta))\) is a solution to the following
  system of equations.
  \begin{align}
    \label{eq:tpe:1}
    G_{\alpha, \beta, t(\beta, y')} \circ \varphi_{\beta}(0, y', 0) &= G_{\alpha, \beta, t(\beta, y')} \circ \varphi_{\beta}(x, y, 0)
    \\
    (\varphi_{\beta})_1 \circ \chi((\varphi_{\beta})^{-1}(0, y', 0), \beta, t(\beta, y')) &= (\varphi_{\beta})_1 \circ \chi((\varphi_{\beta})^{-1}(x, y, 0), \beta, t(\beta, y'))
    \\
    \label{eq:tpe:3}
    (\varphi_{\beta})_2 \circ \chi((\varphi_{\beta})^{-1}(0, y', 0), \beta, t(\beta, y')) &= (\varphi_{\beta})_2 \circ \chi((\varphi_{\beta})^{-1}(x, y, 0), \beta, t(\beta, y'))
  \end{align}
\end{proposition}

Note that the expressions of the
equations~(\ref{eq:tpe:1}\,--\,\ref{eq:tpe:3}) are computed
symbolically by the computer (see page~\pageref{source_code_info} for
more information).  They are not displayed here because of their
length.  Figure~\ref{fig:tpe:contours} shows a contour plot of the HS
transform \(G_{\alpha, \beta, t}\) with
\(\alpha = (0, \frac{1}{2}, \frac{1}{2})\) and \((\beta, t)\) on the
line TPE.

\begin{figure}
  \centering
  \includegraphics{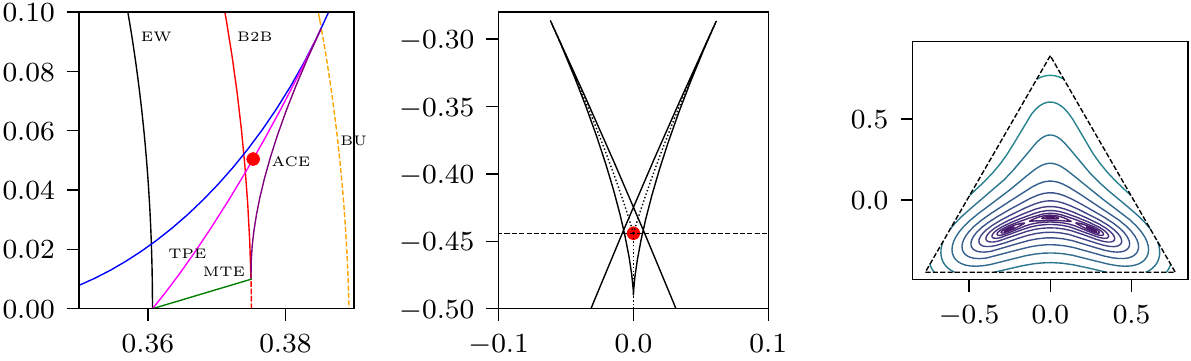}
  \caption{The four-dimensional parameter \((\alpha, \beta, t)\) is
    represented by the two red dots in the two plots on the left.  The
    first of these plots displays a region of the dynamical phase
    diagram and the second plot the respective bifurcation set slice
    clipped to a rectangle near the lower simplex edge which is
    represented by the dashed horizontal line.  The rightmost plot
    shows contour lines of the potential \(G_{\alpha, \beta, t}\) for
    the respective parameter.  As expected for a triple point, the
    contour lines show three equally deep minimizers of the
    potential.}
\label{fig:tpe:contours}
\end{figure}

\begin{proof}
  The system of equations mainly comes from two ingredients: equal
  depth of two minimizers and same end-conditioning \(\alpha\) for
  these two minimizers.  The triple point is characterized by a
  coexistence of three global minimizers and since a triple point
  \(\alpha\) must fulfill the symmetry relation
  \(\alpha_2 = \alpha_3\), we find that it is sufficient to compare
  the two minimizers in the fundamental cell \(m_1 \le m_2 \le m_3\).
  Because \(\alpha_2 = \alpha_3\), we always have one symmetric
  stationary point so that the two minimizers have the coordinates
  \((0, y', 0)\) and \((x, y, 0)\).  Since we now that either
  minimizer is a stationary point, we can use the vanishing of the
  first partial derivative of \(G_{\alpha, \beta, t}\) with respect to
  the \(y\)-coordinate to eliminate the time variable \(t\) from the
  equations.  This yields the function in
  equation~\eqref{eq:tpe:time_fct}.  Using this function we can
  eliminate the variable \(t\) from the equal depth condition and the
  other two equations that require that the minimizers belong to the
  same end-conditioning \(\alpha\).
\end{proof}

\subsection{The beak-to-beak (B2B) line}
\label{sec:b2b_line}

\begin{figure}
  \includegraphics{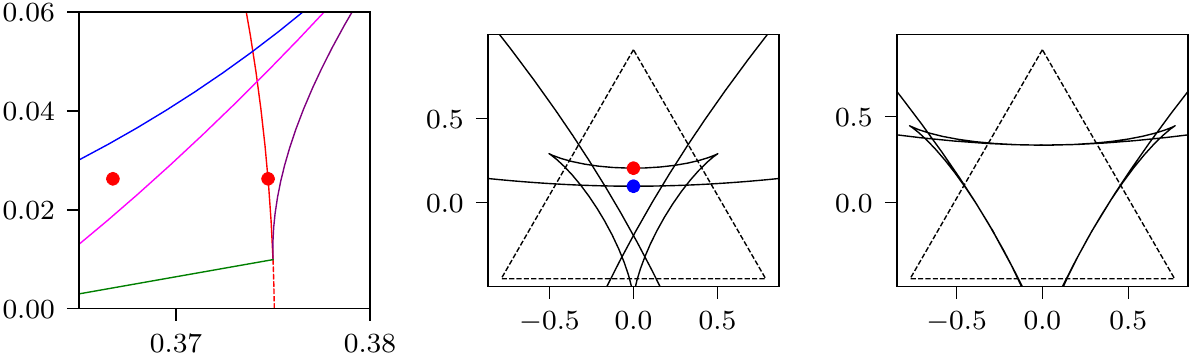}
  \caption{The beak-to-beak mechanism is characterized by the merging
    of two horns of two different pentagrams.  This merging joins two
    connected components of the complement of the bifurcation set
    slice when crossing the red line from right to left.  As can be
    seen in the two rightmost plots, this merging happens on the axis
    of symmetry.  The red dots in the dynamical phase diagram on the
    left mark the time\,--\,temperature pairs that correspond to the
    bifurcation set slices from left to right.  The dots in the
    central plot correspond to the points of the same color in
    Figure~\ref{fig:b2b:alpha1_vs_y}.}
  \label{fig:b2b:slices}
\end{figure}

The beak-to-beak point in the static model is characterized as a cusp
point that lies in a segment from the center of the simplex to one
vertex, that is, for example it has \(y > 0\).  The following
proposition describes the line of beak-to-beak points and a parametric
representation in terms of roots of a cubic polynomial.  Note that,
despite the fact that the line continues to exist for \(\beta > 3\),
the structural behavior of the bifurcation set around the beak-to-beak
point might change in the regime \(\beta > 3\).

\begin{proposition}
  Fix any positive \(\beta\) and \(t\), let \(m\) be a point in \(H\)
  with coordinates \((0, y, 0)\).

  \begin{enumerate}
  \item The point
    \(\alpha = \chi(m, \beta, t)\) is a beak-to-beak point if and only if
    \begin{gather}
      \label{eq:b2b:eq1}
      -(\beta + 6y - 2)(e^{g_t} + 1)e^{-3y} - (\beta - 3y - 1)e^{g_t + 3y} + e^{g_t}(e^{g_t} + 1) = 0
      \\
      \label{eq:b2b:eq2}
      (\beta + 6y - 4)(e^{g_t} + 1)e^{-3y} - (\beta - 3y - 2)e^{g_t +
        3y} = 0
    \end{gather}
  \item The solutions to this system can be parametrized in terms of
    \(s = 3y\) in the form
    \begin{align}
      \label{eq:b2b:beta}
      \beta &= \frac{
        2(s-2)w_\ast(s) + (s+2)(w_\ast(s)-1)e^{2s}
      }{
        (w_\ast(s)-1)e^{2s} - w_\ast(s)
      }
      \\
      g_t &= \log(w_\ast(s) - 1)
    \end{align}
    where \(s > s_\ast \approx 0.66656\) and \(w_\ast(s)\) is the
    unique root in the interval \((2, \infty)\) of the cubic
    polynomial
    \begin{equation}
      \label{eq:b2b:cubic}
      \begin{split}
        (e^{3s} - e^s) \, w^3 - (6s e^{2s} + e^{4s} + 2e^{3s} -
        3e^{2s} - e^s - 2) \, w^2 +
        \\
        (6se^{2s} + 2e^{4s}+3e^{3s} - 3e^{2s} -2e^s) \, w - e^{4s} - 2e^{3s}.
      \end{split}
    \end{equation}
    The positive real number \(s_\ast\) is the unique root in
    \((0, \infty)\) of the function
    \begin{equation}
      \label{eq:b2b:domain_def}
      s \mapsto -12se^{2s} - e^{4s} + 4e^{3s} + 6e^{2s} - 8e^s + 8.
    \end{equation}
  \item The beak-to-beak point enters the simplex for
    \(s = 2/3 > s_\ast\) at which \(\beta = \frac{8}{3}\) and
    \(g_t \approx 0.026481\).
\end{enumerate}

\end{proposition}

\begin{proof}
  From the analysis of the static model \parencite[see][Figure~2,
  rightmost plot of the first row and neighbouring plots for smaller
  or larger \(\beta\)]{KuMe20} we know that the
  beak-to-beak point \((\alpha_\ast, \beta_\ast, t_\ast)\) is such
  that if we fix \(\alpha = \alpha_\ast\) but change the parameters
  \(\beta\) or \(t\) we either find that \(\alpha = \alpha_\ast\) is
  contained in a cell with two minimizers or in a cell with one
  minimizer.  Since \(\alpha_\ast\) lies on the axis of symmetry, we
  know \(\alpha_\ast = \chi(m_\ast, \beta_\ast, t_\ast)\) where
  \(m_\ast\) lies on the axis of symmetry as well, and we find in
  coordinates
  \(\varphi_{\beta}(\alpha_\ast) = (0, v(m_\ast, \beta_\ast, t_\ast),
  0)\), so it suffices to study
  \begin{equation}
    \begin{split}
      v(m, \beta, t) &= (\varphi_{\beta})_2 \circ \chi(m, \beta, t) =
      \\
      &\frac{(\beta + 6y)w e^{-3y} - (\beta - 3y)(w-1)e^{3y}+3(w^2 - w
        + 2)y - \beta}{3(w^2 - w - 2)}
    \end{split}
  \end{equation}
  as a function of the \(y\)-coordinate of \(m\).  As before
  substitute \(w = e^{g_t} + 1\).  In Figure~\ref{fig:b2b:alpha1_vs_y}
  you see a minimum and a maximum collide and form a saddle point.
  This is exactly the beak-to-beak behavior.  The point \((\beta, t)\)
  for which this collision has just happened is given by the vanishing
  of the first and second derivatives of \(v(m, \beta, t)\) with
  respect to the \(y\)-coordinate of \(m\).  Now, the derivatives are
  given by:
  \begin{align}
    \label{eq:b2b:alpha1_vs_y:y}
    \dv{v}{y}\qty(m, \beta, t) &= \frac{
    -(\beta + 6y - 2)we^{-3y} - (\beta - 3y - 1)(w-1)e^{3y} + w^2 - w + 2
    }{
    w^2 - w - 2
    }
    \\
    \label{eq:b2b:alpha1_vs_y:yy}
    \dv[2]{v}{y}\qty(m, \beta, t) &= \frac{
    3(\beta + 6y - 4)we^{-3y} - 3(\beta - 3y - 2)(w-1)e^{3y}
    }{
    w^2 - w - 2
    }
  \end{align}
  Since \(w > 2\), it suffices to consider the numerators of the above
  expressions.  This yields equations~\eqref{eq:b2b:eq1} and
  \eqref{eq:b2b:eq2}.

  Let us now prove the parametric form of the solutions.
  Equation~\eqref{eq:b2b:eq2} is linear in \(\beta\) as long as
  \(e^{g_t} + 1 - e^{g_t + 6y} \neq 0\) and can then be solved for
  \(\beta\) to yield \eqref{eq:b2b:beta} after substituting
  \(w = e^{g_t} + 1\) and \(s = 3y\).  Suppose now
  \(e^{g_t} + 1 - e^{g_t + 6y} = 0\) which is equivalent to
  \(e^{g_t} = \frac{1}{e^{6y} - 1}\).  Equation~\eqref{eq:b2b:eq2}
  would in this case read
  \begin{equation}
    \frac{(9y - 2)e^{3y}}{e^{6y} - 1} = 0
  \end{equation}
  which is only fulfilled for \(y = \frac{2}{9}\).  However, this
  leads to the contradiction
  \(e^{g_t} = \frac{1}{e^{\frac{4}{3}} - 1} < 1\) but \(g_t > 0\).
  Therefore, we can assume that we can solve \eqref{eq:b2b:eq2} for
  \(\beta\).  Plugging this into equation~\eqref{eq:b2b:eq1} we arrive
  at the following fraction of polynomials in \(w\).
  \begin{equation}
    \frac{
      \begin{split}
      (e^{3s} - e^s) \, w^3 &- (6s e^{2s} + e^{4s} + 2e^{3s} -
      3e^{2s} - e^s - 2) \, w^2
      \\
      &+(6se^{2s} + 2e^{4s}+3e^{3s} - 3e^{2s} -2e^s) \, w - e^{4s} - 2e^{3s}
      \end{split}
    }{
      e^s\big((w - 1)e^{2s} - w\big)
    } = 0.
  \end{equation}
  The denominator is not zero because we are able to solve for
  \(\beta\). Thus, it suffices to consider the numerator which yields
  Formula~\eqref{eq:b2b:cubic}.

  We will now discuss the roots larger than 2 of this cubic
  polynomial.  It is convenient to change variables
  \(\theta = w - 2\), so that we are interested in the positive roots
  of the following polynomial:
  \begin{equation}
    \begin{split}
    \theta^3(e^{3s} - e^s)
    &-(6se^{2s} + e^{4s} - 4e^{3s} - 3e^{2s} + 5e^{s} - 2)\theta^2
    \\
    &- (18se^{2s} + 2e^{4s} -7e^{3s} - 9e^{2s} +10e^s - 8)\theta
    \\
    &- 12se^{2s} - e^{4s} + 4e^{3s} + 6e^{2s} - 8e^s + 8
    \end{split}
  \end{equation}
  Using Descartes' rule of signs, we know that the number of positive
  roots is equal to the number of sign changes among consecutive,
  nonzero coefficients of the polynomial or it less than it by an even
  number.  Note that the coefficients in increasing order for \(s = 0\)
  are given by \((9, 12, 3, 0)\).  Therefore we do not find any
  positive roots for very low positive values of \(s\).  The first
  sign changes appears for the coefficient of order zero which yields
  equation~\eqref{eq:b2b:domain_def}.  All of the coefficients except
  the highest order coefficient eventually become negative.  However,
  with increasing \(s\) this happens with increasing order of the
  coefficient so that we have only one sign change between consecutive
  coefficients for each \(s\) larger than \(s_\ast\).  Thus, for all
  \(s > s_\ast\) there exists only one root \(w_\ast(s)\) larger than
  2.
  \begin{figure}
    \centering
    \includegraphics{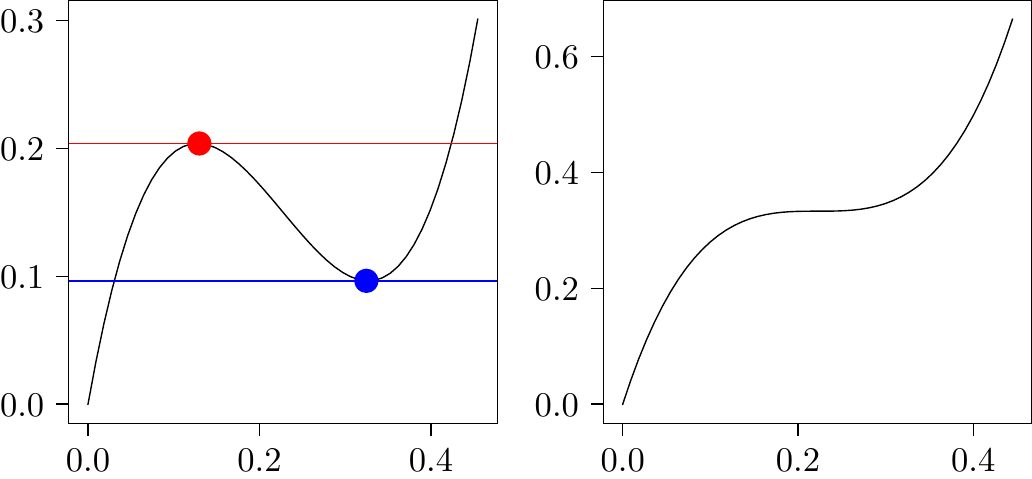}
    \caption{This figure shows how \(v(m, \beta, t)\) behaves as a
      function of the \(y\)-coordinate of \(m\) for
      \(g_t \approx 0.07012\).  In the left plot
      (\(\beta \approx 2.6685\)) you see that there is a region for
      \(v(m, \beta, t)\) such that there exist three solutions to the
      equation \(v(m, \beta, t) = v_0\).  In the right plot
      (\(\beta \approx 2.7267\)) this region is gone.  For any \(v_0\)
      in this region, we find three zeros of the partial derivative of
      the potential with respect to the \(y\)-coordinate of \(m\)
      corresponding to two local minimizers and a saddle point.  The
      red and blue dots correspond to the same dots in the central
      plot of figure~\ref{fig:b2b:slices}}
\label{fig:b2b:alpha1_vs_y}
  \end{figure}
\end{proof}

\subsection{Reentry into Gibbs: the Maxwell triangle exit (MTE) line}

For \(\beta\) in the interval \((\frac{8}{3}, 4\log 2)\) the model
displays recovery as well but due to a different mechanism.  After the
horns of two pentagrams have touched, the Maxwell set which consisted
of three connected components now has become one connected component.
It consists of three straight lines on the axes of symmetry and a
triangle with curved edges.  The model recovers from the
non-Gibbsianness when this triangle completely leaves the unit simplex
which happens on another line in the dynamical phase diagram we call
\emph{Maxwell triangle exit (MTE)}.

\begin{proposition}
  For any \(\beta\) in the interval \((\frac{8}{3}, 4\log 2)\) define
  the function
  \begin{equation}
    \label{eq:mte:w}
    w(\beta, y) = 1 + \frac{(\beta + 6y)e^{-3y}}{\beta - 3y}.
  \end{equation}
  The Maxwell triangle leaves the simplex at
  \(t = t_{\mathrm{MTE}}(\beta) = \frac{1}{3}\log \frac{w(\beta, y) +
    1}{w(\beta, y) - 2}\) where \(y\) in
  \((-\frac{\beta}{6}, \frac{\beta}{3})\) is such that there exists a
  \(y'\) in \((-\frac{\beta}{6}, \frac{\beta}{3})\) and \((y, y')\) is
  a solution of the system
  \begin{align}
    \label{eq:mte:eq1}
    (\beta + 6y)(\beta - 3y')e^{-3y} - (\beta + 6y')(\beta - 3y)e^{-3y'} &= 0
    \\
    \label{eq:mte:eq2}
    -2y - {y^\prime} - \frac{3}{\beta} \Big((y^\prime)^2 - y^2\Big) + \log \frac{\beta}{3}\left(-2 \, {\left(\beta - 3 \, y\right)} e^{3 \, y} - {\left(\beta + 6 \, y\right)} e^{3 \, {y^\prime}}\right) &= 0
  \end{align}
\end{proposition}

Before we come to the proof, let us remark the following: Of course,
it is impractical to solve this system by hand.  However, for fixed
\(\beta\) we can show the zeros of the left-hand sides of both
equations.  Figure~\ref{fig:mte:zeros} shows them in the relevant
rectangle
\((-\frac{\beta}{6}, \frac{\beta}{3}) \times (-\frac{\beta}{6},
\frac{\beta}{3})\).  The line as depicted in the dynamical phase
diagram is obtained via a numerical solution of this system of
equations.

\begin{figure}
  \centering \includegraphics{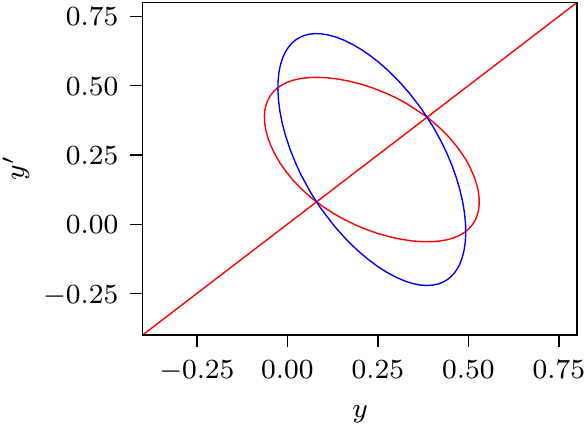}
  \caption{The zeros of the left-hand sides of the two
    equations~\eqref{eq:mte:eq1} and \eqref{eq:mte:eq2} for
    \(\beta = 2.8\).  The red curve corresponds to the solutions of
    \eqref{eq:mte:eq1} and the blue curve to the solutions of
    \eqref{eq:mte:eq2}.  The intersection of the red curve with
    diagonal is of course a trivial solution and not the one we are
    looking for.}
\label{fig:mte:zeros}
\end{figure}

\begin{proof}
  Let \(m = \varphi_{\beta}^{-1}(0, y, 0)\) be any point on the axis
  of symmetry with \(m_2 = m_3\).  This point is mapped to
  \(\alpha = (1, 0, 0)\) by the catastrophe map
  \(\chi(\cdot, \beta, t)\) if and only if
  \begin{equation}
    \frac{6y}{\beta} + 1 - \frac{3(w-1)e^{3y}}{(w-1)e^{3y} + 2} = 0
  \end{equation}
  which is the equation \(\pdv{G_{\alpha, \beta, t}}{y} = 0\) where
  \(\alpha = (1, 0, 0)\) and we have substituted \(w = e^{g_t} + 1\).
  Solving this equation for \(w\) we find two solutions one of which
  is positive. This yields \eqref{eq:mte:w}.

  Let \(m' = \varphi_{\beta}^{-1}(0, y', 0)\) be any point on the same
  axis of symmetry.  The value of \(G_{\alpha, \beta, t}\) at these
  two points \(m\) and \(m'\) are equal if and only if
  \(G_{\alpha, \beta, t} - G_{\alpha, \beta, t} = 0\).  Plugging in
  \(t = \frac{1}{3}\log\frac{w(\beta, y) + 1}{w(\beta, y) - 2}\) and
  \(\alpha = (1, 0, 0)\) yields \eqref{eq:mte:eq2}.
  Equation~\eqref{eq:mte:eq1} comes from the fact that \(m\) and
  \(m'\) are stationary points that belong to the same time variable
  \(t\), that is, \(w(\beta, y) - w(\beta, y') = 0\).  If we multiply
  this equation by \((\beta - 3y)(\beta - 3y')\) we arrive at
  \eqref{eq:mte:eq1}.
\end{proof}

\section{Loss of the Gibbs property without recovery}
\label{sec:loss_only}

If \(\beta\) lies in the interval \((4\log 2, 3)\), the model displays
the loss of the Gibbs property without recovery.  This is due to the
uniform distribution which becomes bad after a sharp transition time
and stays bad forever.  This behavior is analogous to the behavior in
the static model described by the Ellis-Wang theorem
\parencite{ElWa90}.

\subsection{The Ellis-Wang (EW) line}

The static model has a phase-coexistence of four states at inverse
temperature~\(4\log 2\) in zero field \parencite{ElWa90}.
The first layer model as discussed in this paper has a whole line of
such points which we refer to as Ellis-Wang points.

\begin{proposition}
  Suppose \(\alpha = (\frac{1}{3}, \frac{1}{3}, \frac{1}{3})\), that
  is, it represents the uniform distribution.
  \begin{enumerate}
  \item The HS transform \(G_{\alpha, \beta, t}\) has a point of
    phase-coexistence with four global minimizers if and only if there
    exists a solution \((s, \beta, t)\) to the following system of
    equations.
    \begin{align}
      \label{eq:ew:eq1}
      \frac{3y}{\beta} + \frac{1}{e^{3y + g_t} + 2} - \frac{e^{3y}}{e^{3y} + e^{g_t} + 1}&= 0
      \\
      \label{eq:ew:eq2}
      3y\qty(1 + \frac{3y}{\beta}) + \log\frac{
      (e^{g_t} + 2)^3
      }{
      (e^{g_t} + 1 + e^{3y})^2 (e^{3y + g_t} + 2)
      } &= 0
    \end{align}
  \item The solutions to the above systems can be parametrized in
    terms of \(s = 3y\) given via
    \begin{align}
      \label{eq:ew:beta}
      \beta &= \frac{
              s\big(e^s(w_\ast(s) - 1) + 2\big)(e^s + w_\ast(s))
              }{
              (e^s - 1)(w_\ast(s) e^s + w_\ast(s) - e^s)
              },
      \\
      g_t &= \log(w_\ast(s) - 1)
    \end{align}
    where \(s > 2\log 2\) and \(w_\ast(s)\) is the unique
    zero in \((2, \infty)\) of
    \begin{equation}
      \label{eq:ew:zero_problem}
      w \mapsto s\left( 1 + \frac{(e^s - 1)(we^s - e^s + w)}{(w+e^s)(we^s-e^s+2)}\right) + \log\frac{
        (w+1)^3}{
        (w+e^s)^2 (we^s - e^s + 2)
        }.
    \end{equation}
  \end{enumerate}
\end{proposition}

\begin{proof}
  First, let us derive the system of
  equations~(\ref{eq:ew:eq1}--\ref{eq:ew:eq2}).  Since \(\alpha\) has
  the full symmetry, that is, it is invariant under any permutation of
  \(S_3\), it suffices to consider the equal-depth of the central
  minimum \(m_0\) with one of the three outer ones denoted by \(m\).
  In the following, we assume \(m_2 = m_3\).  The relative difference
  between the values is given by
  \begin{equation}
    \begin{split}
      G_{\alpha, \beta, t}(m) - G_{\alpha, \beta, t}(m_0) = y &+ \frac{3y^2}{\beta} - \frac{1}{3} \log(e^{g_t + 3y} + 2)
      \\
      &- \frac{2}{3} \log(e^{g_t} + e^{3y} + 1) + \log(e^{g_t} + 2).
    \end{split}
  \end{equation}
  By collecting the logarithmic terms and multiplying the equation by
  3 we find \eqref{eq:ew:eq2}.  Equation~\eqref{eq:ew:eq1} comes from
  the fact that \(m\) is a stationary point.  So we calculate the
  relevant partial derivative
  \begin{equation}
    \pdv{G_{\alpha, \beta, t}}{y} = \frac{6y}{\beta} + 1 - \Gamma_{1,1} - 2\Gamma_{2, 1}
  \end{equation}
  where \(\Gamma_{b, a} = \Gamma_{b, a}(\beta m, t)\).  The partial
  derivative with respect to the \(x\)-coordinate of \(m\) vanishes
  because of symmetry.  Plugging in the expressions for
  \(\Gamma_{1,1}\) and \(\Gamma_{2, 1}\) yields \eqref{eq:ew:eq1}.

  Now, let us come to the parametrization.
  Equation~\eqref{eq:ew:beta} follows by substituting
  \(w = e^{g_t} + 1\) and \(s = 3y\) in equation~\eqref{eq:ew:eq1} and
  solving for \(\beta\) which is possible since \(s \neq 0\).
  Plugging this into equation~\eqref{eq:ew:eq2} and making the same
  substitutions we find \eqref{eq:ew:zero_problem}.  Note that
  \(w_\ast(s)\) is increasing with \(s\) and that the solution of
  \(w_\ast(s) = 2\) is \(s = 2\log 2\).  For lower values of \(s\)
  \eqref{eq:ew:zero_problem} has no zeros larger than two.
\end{proof}

\subsection{The elliptic umbilics (EU) line}

In the static model there is a special point called elliptic umbilic.
This catastrophe at the center of the unit simplex is responsible for
the fact that the central minimum changes to a maximum.  In the
dynamical model\,--\,due to the additional parameter \(g_t\)\,--\,we
have a whole line of these points.  This line we call the line of
elliptic umbilics (EU).

\begin{proposition}
  For each \(\beta \ge 3\) define the function
  \begin{equation}
    \label{eu:w}
    w(\beta) = \beta - 1 + \sqrt{\beta(\beta - 3)}.
  \end{equation}
  Fix some \(\beta \ge 3\) and let
  \(\alpha = (\frac{1}{3}, \frac{1}{3}, \frac{1}{3})\) and
  \(t = \frac{1}{3} \log\frac{w(\beta) + 1}{w(\beta) - 2}\).  Then:
  \begin{enumerate}
  \item The Hessian \(G_{\alpha, \beta, t}''(m)\) at
    \(m = (\frac{1}{3}, \frac{1}{3}, \frac{1}{3})\) has a double zero
    eigenvalue.
  \item The Taylor expansion of \(G_{\alpha, \beta, t}\) at
    \(m = (\frac{1}{3}, \frac{1}{3}, \frac{1}{3})\) for \(\beta = 3\)
    (and therefore \(g_t = 0\)) up to the third order is given by
    \begin{equation}
      \label{eq:eu:expansion}
      x^2y - \frac{1}{3}y^3 + \frac{1}{2} z^2 - \log 3 - \frac{1}{2}.
    \end{equation}
  \end{enumerate}
\end{proposition}

\begin{proof}
  First, we check that the Hessian has a double zero eigenvalue.  Let
  \(\alpha\) equal \((\frac{1}{3}, \frac{1}{3}, \frac{1}{3})\) and
  consider the Hessian of \(G_{\alpha, \beta, t}\) at
  \(m = (\frac{1}{3}, \frac{1}{3}, \frac{1}{3})\).  With the same
  arguments as in the proof of Proposition~\ref{prop:sce}, we find
  that the Hessian is diagonal.  Furthermore, since \(\alpha\) and
  \(m\) have the full symmetry, the two second order partial
  derivatives \(\pdv[2]{G_{\alpha, \beta, t}}{y}\) and
  \(\pdv[2]{G_{\alpha, \beta, t}}{x}\) are equal.  Let us consider the
  partial derivative with respect to \(y\).
  \begin{equation}
    \begin{split}
      \pdv[2]{G_{\alpha, \beta, t}}{y}
      &= \frac{6}{\beta}
      - 3\qty(\Gamma_{1,1} - \Gamma_{1,1}^2
      + 2(\Gamma_{2, 1} - \Gamma_{2, 1}^2))
      \\
      &= \frac{6}{\beta} - 3\qty(\frac{e^{g_t}}{e^{g_t} + 2}
      -\frac{e^{2g_t}}{(e^{g_t} + 2)^2}
      + \frac{2}{e^{g_t} + 2} - \frac{2}{(e^{g_t} + 2)^2}
      )
      \\
      &= \frac{6}{\beta} - 3\qty(1
      -\frac{(w-1)^2 + 2}{(w + 1)^2}
      )
      \\
      &= 6\frac{w^2 + 2(1 - \beta)w + 1 + \beta}{\beta(w + 1)^2}
      \\
    \end{split}
  \end{equation}
  where \(\Gamma_{b, a} = \Gamma_{b, a}(\beta m, t)\) and we have
  substituted \(w = e^{g_t} + 1\).  Setting this equal to zero and
  solving for \(w\) yields \eqref{eu:w} since the other root of the
  quadratic polynomial in the numerator is always less than two.

  Now we come to (b).  Plugging \(\beta = 3\) and \(g_t = 0\) into the
  HS transform and writing it in the \((x, y, z)\)-coordinates we
  arrive at
  \begin{equation}
    \label{eq:eu:hs}
    \begin{split}
      G_{\alpha, \beta, t}(m) = \frac{3}{2} \langle m, m\rangle - \log
      \sum_{a=1}^3 e^{3m_a} = x^2 + y^2 + \frac{1}{2}z^2 + \sqrt{3}x
      \\
      + y - \frac{1}{2} - \log(1 + e^{2\sqrt{3}x} + e^{\sqrt{3}x +
        3y}).
    \end{split}
  \end{equation}
  Using the Taylor expansion of the logarithm and the exponential
  function, \eqref{eq:eu:expansion} follows by an elementary
  computation.  Note that \eqref{eq:eu:hs} is actually the HS
  transform of the static Potts model.
\end{proof}

Using symbolic computation with the help of a computer, it is also
possible to obtain a Taylor expansion for every pair \((\beta, g_t)\)
on the Elliptic umbilic line.  Because of symmetry, the
\(\beta\)-dependent coefficients of \(x^2y\) and \(y^3\) differ only
by a factor of \(-\frac{1}{3}\).  This means that for any
\((\beta, g_t)\) on the Elliptic umbilic line the potential
\(G_{\alpha, \beta, t}\) with \(\alpha\) representing the uniform
distribution has the following Taylor expansion up to order three
around the simplex center.
\begin{equation}
  \frac{A_1(\beta)}{A_2(\beta)}\qty(x^2y - \frac{1}{3} y^3) + \frac{3}{2\beta} z^2
  -\frac{1}{6}\beta - \log(\beta + \sqrt{\beta(\beta - 3)})
\end{equation}
The functions \(A_1(\beta), A_2(\beta)\) are given as follows:
\begin{equation}
  \begin{split}
    A_1(\beta) = {}&7077888 \, \beta^{10} - 107937792 \, \beta^{9} +
    700710912 \, \beta^{8} - 2523156480 \, \beta^{7}\\
    &+ 5502422016 \, \beta^{6} - 7445737728 \, \beta^{5} + 6152433408
    \, \beta^{4}\\
    &- 2930719968 \, \beta^{3} + 712130940 \, \beta^{2} - 67493007 \, \beta + 1062882\\
    &+ 27 \, B(\beta) \sqrt{\beta(\beta - 3)}
    \end{split}
\end{equation}

\begin{equation}
  \begin{split}
    B(\beta) = {}&262144 \, \beta^{9} - 3604480 \, \beta^{8} + 20840448
    \, \beta^{7} - 65802240 \, \beta^{6}\\
    &+ 123282432 \, \beta^{5} - 139366656 \, \beta^{4} + 92378880 \,
    \beta^{3} - 33102432 \, \beta^{2}\\
    &+ 5380020 \, \beta - 255879
  \end{split}
\end{equation}

\begin{equation}
  \begin{split}
    A_2(\beta) = {}&1048576 \, \beta^{12} - 16515072 \, \beta^{11} +
    111476736 \, \beta^{10} - 421134336 \, \beta^{9}\\
    &+ 975421440 \, \beta^{8} - 1426553856 \, \beta^{7} + 1307674368
    \, \beta^{6}\\
    &- 720555264 \, \beta^{5} + 218245104 \, \beta^{4} - 30311820 \,
    \beta^{3}\\
    &+ 1240029 \, \beta^{2} + C(\beta) \sqrt{\beta(\beta - 3)}
  \end{split}
\end{equation}

\begin{equation}
  \begin{split}
    C(\beta) = {}&1048576 \, \beta^{11} - 14942208 \, \beta^{10} +
    90243072 \, \beta^{9} - 300810240 \, \beta^{8}\\
    &+ 603832320 \, \beta^{7} - 747242496 \, \beta^{6} + 560431872 \,
    \beta^{5} - 240185088 \, \beta^{4}\\
    &+ 51963120 \, \beta^{3} - 4330260 \, \beta^{2} + 59049 \, \beta
  \end{split}
\end{equation}

\paragraph{Acknowledgements} Daniel Meißner has been supported by the
German Research Foundation (DFG) via Research Training Group 2131
\emph{High dimensional Phenomena in Probability -- Fluctuations and
  Discontinuity}.  We also acknowledge the work of RUB IT-Services who
provide the applications
server\footnote{\url{https://www.it-services.ruhr-uni-bochum.de/services/issi/applikationserver_apps.html}}
that we used for our numerical results.

\label{source_code_info}
\paragraph{Source code for symbolical computation}
As we have already mentioned, since the expressions showing up in the
characterizations of the transition lines are long, we use the
SageMath package for our symbolic computation.  In fact the
expressions are so long that evaluating those expressions using the
Sage interpreter is very time-consuming.  Therefore, we use the code
generation facilities of the \texttt{sympy} library to generate C code
so that we can do the numerical computation in C.  We have included
the file \texttt{potts-numerics-1.0.tar.gz} in the Electronic
Supplemental Material (ESM) which contains the code and further
information on how to use it.  With this package you can, for example,
create high-resolution plots of the functions involved in the
computation of the lines BU, TPE and ACE, or generate bifurcation set
slices and Maxwell set slices.

\printbibliography
\end{document}